\documentclass{amsart}

\usepackage{amsmath,amssymb,xspace}
\usepackage{tikz-cd,etoolbox}
\usepackage{a4wide}
\usepackage[textsize=tiny]{todonotes}
\usepackage[colorlinks]{hyperref} 
   
\theoremstyle{plain}
\newtheorem{theorem}{Theorem}[section]
\newtheorem{proposition}[theorem]{Proposition}

\newtheorem{lemma}[theorem]{Lemma}

\theoremstyle{definition}
\newtheorem{definition}[theorem]{Definition}
\newtheorem{remark}[theorem]{Remark}

\numberwithin{equation}{section}

\def\varpi{t}

\def\sgn{{\rm sgn}}

\def\im{\, {\rm Im}\, \tau}

\def\({\left(}
\def\){\right)}
\def\[{\left[}
\def\]{\right]}
\def\<{\left\langle}
\def\>{\right\rangle}

\newcommand{\de}{\mathrm{d}}
\newcommand{\De}{\mathrm{D}}

\newcommand{\La}{\Lambda}
\newcommand{\Si}{\Sigma}

\DeclareSymbolFont{AMSa}{U}{msa}{m}{n}
\DeclareSymbolFont{AMSb}{U}{msb}{m}{n}
\DeclareMathSymbol{\fieldR}{\mathalpha}{AMSb}{"52}

\newcommand{\cI}{\mathcal{I}}

\newcommand{\cA}{\mathcal{A}}

\newcommand{\nn}{\nonumber}

\newcommand{\eps}{\varepsilon}

\newcommand{\IC}{\mathbb{C}}
\newcommand{\IZ}{\mathbb{Z}}
\newcommand{\IH}{\mathbb{H}}

\newcommand{\IP}{\mathbb{P}}

\def\bea{\begin{eqnarray}}
\def\eea{\end{eqnarray}}
\def\be{\begin{equation}}
\def\ee{\end{equation}}
\def\ba{\begin{align}}
\def\ea{\end{align}}
\def\bse{\begin{subequations}}
\def\ese{\end{subequations}}

\def\ba{\bar a}
\def\bk{\bar k}

\def\bZ{\bar Z}

\newcommand{\cB}{\mathcal{B}}

\def\cij#1{c}
\def\ci#1{c}

\def\XXint#1#2#3{{\setbox0=\hbox{$#1{#2#3}{\int}$}
\vcenter{\hbox{$#2#3$}}\kern-.5\wd0}}

\def\gamD#1{\tilde\gamma}

\DeclareMathOperator{\ch}{ch}
\DeclareMathOperator{\rk}{rk}

\def\cl0{\tilde c_0}

\newcommand{\bfK}{{\boldsymbol K}}

\def\mat#1{\ensuremath{#1}\xspace}
\def\dmat#1#2{\gdef#1{\mat{#2}}}
\def\dmata#1#2{\csdef{#1}{\mat{#2}}}
\def\oper#1{\dmata{#1}{\operatorname{#1}}}
\def\defbb#1{\dmata{b#1}{\mathbb{#1}}}
\def\defcal#1{\dmata{c#1}{\mathcal{#1}}}
\def\defbf#1{\dmata{bf#1}{\mathbf{#1}}}

\def\redef#1{\csletcs{t@#1}{#1}\dmata{#1}{\csuse{t@#1}}}

\forcsvlist\defbb{A,B,C,D,E,F,G,H,I,J,K,L,M,N,O,P,Q,R,S,T,U,V,W,X,Y,Z}
\forcsvlist\defcal{A,B,C,D,E,F,G,H,I,J,K,L,M,N,O,P,Q,R,S,T,U,V,W,X,Y,Z}
\defbf{K}
\forcsvlist\redef{Phi,mu,infty}
\dmat\al\alpha
\dmat\ga\gamma
\dmat\ta\tau
\dmat\vi\phi
\dmat\hi\chi
\dmat\ksi\xi
\dmat\si\sigma
\def\fM{\mathfrak M}
\dmat\de\delta
\dmat\De\Delta
\dmat\Ga\Gamma
\dmat\rM{{\mathrm M}}
\dmat\rG{{\mathrm G}}
\dmat\Om\Omega
\dmat\Omb{\bar\Om}
\dmat\bk{\Bbbk}
\dmat\la{\lambda}

\forcsvlist\oper{Ext,Hom,td,Exp,coker,Coh,Vect,ch,Aut,Var,IH,im,Id,MHM,St,Exp,Log}
\forcsvlist\oper{IE,IP,DT,Rep,MHS,Spec,Gr}
\def\lHom{\underline\Hom}

\def\rbr#1{\left(#1\right)}
\def\ang#1{\langle#1\rangle}
\def\set#1{\left\{#1\right\}}
\def\sets#1#2{\left\{\left.#1\ \right\vert#2\right\}}

\def\nn#1{\lVert#1\rVert}
\def\eq#1{\begin{equation}#1\end{equation}}
\def\eql#1#2{\begin{equation}\label{#2}#1\end{equation}} 
\def\over#1#2{{\substack{#1\\#2}}} 
\def\pser#1{[\![#1]\!]}

\def\bop{\bigoplus}
\def\oh{\frac12}
\def\ts{\otimes}
\def\vir{{\rm vir}}
\def\even{{\rm even}}
\def\inv{^{-1}}
\def\lb{\underline}

\def\mto{\mapsto}

\def\xto{\xrightarrow}
\def\iso{\simeq}

\dmat\one{\mathbf1}
\dmat\iff{\Longleftrightarrow}
\def\what{\widehat}

\def\sbs{\subset}

\def\xx{\times}
\def\pt{\mathbf{pt}}
\def\IC{\operatorname{{IC}}}
\def\ICV{\IC^{\mathrm{vir}}}

\AtEndEnvironment{remark}{\hfill$\lozenge$}

\RequirePackage{todonotes}

\def\mot{{\mathrm{mot}}}
\def\mhm{{\mathrm{mhm}}}
\def\fr{{\mathrm{f}}}
\def\fm#1{M^\fr_{#1}} 
\def\mm#1{M_{#1}} 
\def\fs#1{\fM^\fr_{#1}} 
\def\ms#1{\fM_{#1}} 
\def\kvar#1{\bfK^{\mathrm{mot}}(#1)}
\def\hkvar#1{\widehat\bfK^{\mathrm{mot}}(#1)}
\def\kmhm#1{\bfK^{\mathrm{mhm}}(#1)}
\def\hkmhm#1{\widehat\bfK^{\mathrm{mhm}}(#1)}
\def\var#1{\mathrm{Var}/#1}
\def\St#1{\mathrm{St}/#1}

\def\tpdf#1#2{\texorpdfstring{#1}{#2}} 
\def\hm{H} 
\def\cf{cf.~}   
\oper{Bun}   
\def\omn{\Om^{\mathrm{num}}} 
  
\begin{document}    
\title{Intersection cohomology of moduli spaces of sheaves on surfaces}

\author{Jan Manschot}
\author{Sergey Mozgovoy}

\address{School of Mathematics, Trinity College Dublin, Ireland\newline\indent
Hamilton Mathematics Institute, Ireland}

\email{manschot@maths.tcd.ie}
\email{mozgovoy@maths.tcd.ie}

\begin{abstract}
We study intersection cohomology of moduli spaces of semistable vector
bundles on a complex algebraic surface. Our main result relates
intersection Poincar\'e polynomials of the moduli spaces to
Donaldson-Thomas invariants of the surface. In support of this result,
we compute explicitly intersection Poincar\'e polynomials for sheaves
with rank two and three on ruled surfaces.
\end{abstract}

\maketitle

\tableofcontents  

\section{Introduction}
Let $S$ be a complex projective surface and $J$ be a polarization on $S$.
For any $\ga\in H^\even(S)$, let $\mm\ga$ (resp.\ $\ms\ga$) denote the moduli space (resp.\ stack) of Gieseker $J$-semistable sheaves having Chern character \ga.
The moduli space $\mm\ga$ is a projective, possibly singular, variety. In this paper we will prove a relation between intersection cohomologies of $\mm\ga$ and invariants of $\ms\ga$
under certain technical conditions on the polarization. 

To state our result more precisely, let $P(X)\in\bZ[y]$ be the
(motivic) Poincar\'e polynomial for an algebraic variety $X$ (see~\S\ref{sec:E-pol}). 
It is defined for a smooth projective $X$ by the formula
$$P(X)=\sum_{n\ge0}\dim H^n(X)(-y)^n,$$
and then is extended to arbitrary $X$ by additivity with respect to complements. We can represent the stack $\ms\ga$ as a global quotient $R_\ga/G_\ga$, where $G_\ga$ is a general linear group \cite[\S4.3]{huybrechts_geometry},
and define 
$$P(\ms\ga)=P(R_\ga)/P(G_\ga)\in\bQ(y).$$

Let the polarization $J$ and surface $S$ be such that the following two conditions are satisfied:
\begin{enumerate}
\item[(A)]  $J\cdot K_S<0$, implying that the category of semistable sheaves with a fixed reduced Hilbert polynomial (see \S\ref{sec:semist}) has a vanishing second \Ext.
\item[(B)] $J$ is generic (see Remark \ref{rem:generic}), implying that if $E,F$ have equal reduced Hilbert polynomials then $\hi(E,F)=\hi(F,E)$.
\end{enumerate}

Under these conditions, generating functions of $P(\ms\ga)$ have been determined
for sheaves with small rank and $S$ a rational or ruled surface \cite{Gottsche1990, Manschot2011,
  Manschot:2011ym, manschot_sheaves, mozgovoy_invariants, Yoshioka1994, Yoshioka1995}. Through the
Hitchin-Kobayashi correspondence \cite{MR1370660} and 
Donaldson-Uhlenbeck-Yau theorem \cite{donaldson_anti,
  CPA:CPA3160390714},  these generating functions are of interest for
the study of Yang-Mills theories. In particular, the partition function of topologically twisted $\mathcal{N}=4$ supersymmetric
Yang-Mills theory localizes on Hermitian-Yang-Mills connections \cite{Vafa:1994tf}, and it equals the generating
function of Euler numbers $\chi(\mm\ga)$ of $\mm\ga$ if $\ga$ is indivisible. In such cases, semi-stability implies stability and $P(\ms\ga)$ is related to $P(\mm\ga)$ by 
$$P(\ms\ga)=\frac{P(\mm\ga)}{y^2-1}.$$

In the following, we consider arbitrary $\ga$. 
To state our main result, recall that for any
algebraic variety $X$ the  intersection
Poincar\'e polynomial $\IP(X)$ is defined by 
$$\IP(X)=\sum_{n}\dim\IH^n(X)(-y)^n,$$
where $\IH^*(X)$ are intersection cohomologies of $X$. Our main result relates the virtual Poincar\'e
functions $P(\ms\ga)$ and $\IP(\mm\ga)$.

\begin{theorem}\label{thm1}
Let $J$ satisfy the conditions (A) and (B) above.
Then
$$
1+\sum_{p_J(\ga)=p}(-y)^{-\dim\ms\ga}P(\ms\ga)z^\ga
=\Exp\rbr{
\frac{\sum_{p_J(\ga)=p}(-y)^{-\dim M_\ga}\IP(M_\ga)z^\ga}{y\inv-y}}
$$
where $\dim\ms\ga=\dim\mm\ga-1=-\hi(\ga,\ga)$,
the sums run over all \ga with a fixed reduced Hilbert polynomial~$p$ and \Exp is a plethystic exponential \eqref{eq:Exp} defined for $f(y,z)=\sum_\ga f_\ga(y)z^\ga$ as
$$\Exp(f)=\exp\bigg(\sum_{n\ge1}\frac1nf(y^n,z^n)\bigg).$$
\end{theorem}

The conditions (A) and (B) are in particular satisfied for the projective plane $\mathbb{P}^2$. For rank $2$ sheaves on $\mathbb{P}^2$, $\IP(\mm\ga)$ were determined  by Yoshioka \cite[Remark
4.6]{Yoshioka1995} extending work of Kirwan on moduli spaces of
rank $2$ vector bundles on Riemann surfaces \cite{Kirwan01091986}. More recently, $\IP(\mm\ga)$ were determined for rank $3$ and $4$ sheaves on $\mathbb{P}^2$ \cite{Manschot:2011ym, manschot_sheaves}. In further support of Theorem \ref{thm1}, Section \ref{explicitresults} provides $\IP(\mm\ga)$ for rank $2$ and $3$ sheaves on ruled surfaces. For two ruled surfaces, we show that for a specific non-generic polarization, $J=-K_S$, Theorem \ref{thm1} continues to hold.  
Note that the conditions (A) and (B) exclude the case of K3 surfaces which have a $2$-Calabi-Yau category of coherent sheaves. Although our approach fails in this case, we expect that a similar relation between intersection cohomologies and stack invariants should be true.


Let us now reformulate the above result on the level of mixed Hodge structures and $E$-polynomials.
For any algebraic variety $X$, we can consider the cohomology with compact support $H^*_c(X,\bQ)$ as an element in $K_0(\MHS)$ and define the $E$-polynomial of $X$ by taking the Hodge-Euler polynomial of $H^*_c(X,\bQ)$
\eqref{E-pol}
$$E(X)=E(H^*_c(X,\bQ))=\sum_{p,q,n}(-1)^n h^{p,q}(H^n_c(X,\bQ))u^pv^q\in\bZ[u^{\pm1},v^{\pm1}].$$
Furthermore, define
$$E(\ms\ga)=E(R_\ga)/E(G_\ga),$$
$$\bL=E(\bA^1)=E(\bQ(-1)[-2])=uv.$$
The (motivic) Poincar\'e polynomial is equal to $P(X,y)=E(X,y,y)$.

On the other hand, for any quasi-projective variety $X$ of dimension $d$, let $\IC_X$ be its intersection complex \cite[\S1.13]{saito_introduction} (considered as a pure Hodge module of weight $d$) and let
$$\IH^*(X)=H^*_c(X,\IC_X)[-d]$$
be its intersection cohomology (considered as an object in $D^b\MHS$ or as an element in $K_0(\MHS)$).
Define the intersection $E$-polynomial of $X$ by taking the Hodge-Euler polynomial \eqref{E-pol}
$$\IE(X)=E(\IH^*(X))=\sum_{p,q,n}(-1)^n h^{p,q}(\IH^n(X))u^pv^q.$$
Note that if $X$ is projective, then $\IH^*(X)$ is pure of weight zero,
and the intersection Poincar\'e polynomial is equal to $\IP(X,y)=\IE(X,y,y)$.

\begin{theorem}\label{thm2}
Let $J$ satisfy the conditions (A) and (B) above.
Then
$$1+\sum_{p_J(\ga)=p}\bL^{-\oh\dim\ms\ga}E(\ms\ga)z^\ga
=\Exp\rbr{\frac{\sum_{p_J(\ga)=p}\bL^{-\oh\dim M_\ga}\IE(M_\ga)z^\ga}{\bL^\oh-\bL^{-\oh}}},
$$
where $\bL^\oh=-(uv)^\oh$.
\end{theorem}


Note that if $\ga$ is indivisible, then $\mm\ga$ consists of stable sheaves and is smooth.
In this case we obtain from the theorem
$$E(\ms\ga)=\frac{\IE(\mm\ga)}{\bL-1}$$
which is straightforward as $\IE(\mm\ga)=E(\mm\ga)$ and all stabilizers of objects in $\ms\ga$ are isomorphic to $\bC^*$.

The above results can be also formulated in terms of Donaldson-Thomas invariants $\Om_\ga=\Om_\ga(u,v)$ defined by the formula (see (\ref{I to Omb}, \ref{bOmtoOm}) for an explicit expression)
\eq{1+\sum_{p_J(\ga)=p}\bL^{-\dim\ms\ga}E(\ms\ga)z^\ga
=\Exp\rbr{\frac{\sum_{p_J(\ga)=p}\Om_\ga z^\ga}{\bL^\oh-\bL^{-\oh}}}
}
Then the above theorem can be simply written in the form
\eq{\Om_\ga=\bL^{-\oh\dim \mm\ga}\IE(\mm\ga).}

The idea of the proof of the above theorems goes back to \cite{meinhardt_donaldson} (see also \cite{mozgovoy_intersection}, \cite{meinhardt_donaldsona}).
We introduce a smooth moduli space $\fm\ga$ of framed vector bundles (see \S\ref{sec:framed}) which is equipped with a projective map $\pi:\fm\ga\to \mm\ga$.
Then we analyze the intersection complex of $\mm\ga$ by studying the pushforward with respect to $\pi$ of the intersection complex on $\fm\ga$.




One may wonder if a similar result can be proved for the moduli spaces of Mumford (also called~\mu-) semistable sheaves on a surface.
On the one hand, there are no technical difficulties.
If $J\cdot K_S<0$ then the category of Mumford semistable sheaves with a fixed slope has a vanishing second $\Ext$  (see Lemma \ref{lm:hereditary}).
And if a polarization is generic then sheaves having the same slope satisfy $\hi(E,F)=\hi(F,E)$ (see Remark \ref{rem:generic}).
On the other hand, one can show that although the moduli spaces (and stacks) of Gieseker and Mumford semistable sheaves are different in general, their Donaldson-Thomas invariants coincide (see Theorem \ref{Om G and S}).
Therefore we can not expect to get any new invariants from the Mumford semistable sheaves.
The reason for this phenomenon is that the moduli space
$\mm\ga^{\mu,s}$ of Mumford stable sheaves is not dense in the moduli space
$\mm\ga^{\mu}$ of Mumford semistable sheaves in general.
Indeed, $\mm\ga^{\mu,s}$ is contained in the moduli space $\mm\ga$ of Gieseker semistable sheaves which is projective and therefore closed in $\mm\ga^{\mu}$.


The paper is organized as follows. Section \ref{preliminaries}
reviews aspects of sheaves on surfaces, $\lambda$-rings and mixed
Hodge structures. Section \ref{result} proves the main results, Theorems \ref{thm1} and \ref{thm2}. Section
\ref{generating} discusses properties of Donaldson-Thomas invariants and their
generating functions. This is applied in Section \ref{explicitresults}
to determine $\IP(\mm\ga)$ explicitly for sheaves with ranks $2$ or $3$ on
a few ruled surfaces.


\section{Preliminaries}
\label{preliminaries}
Let $S$ be an algebraic surface with the canonical class $K_S$.
Given a coherent sheaf $F$ on $S$, let $c_1$ and $c_2$ be its first and second Chern classes respectively and let $\ga=(r,\ga_1,\ga_2)=\ch F$ be its Chern character (so that $\ga_1=c_1$ and $\ga_2=\oh c_1^2-c_2$).

\subsection{Hirzebruch-Riemann-Roch theorem}
By the Hirzebruch-Riemann-Roch theorem, we have 
\eq{\hi(F)=\hi(S,F)=\int_S \ga\cdot\td(S)
=\ga_2-\oh K_S\ga_1+\hi(O_S)r,}
where the Todd class $\td(S)$ is defined by
$$\td(S)=1-\oh K_S+\hi(O_S)=1-\oh K_S+\frac1{12}\rbr{K_S^2+e(S)}.$$
Applying this to $\lHom(F,F')=F^*\ts F'$ (assuming that $F$ is a vector bundle), we obtain
\eq{
\hi(\ga,\ga'):=\hi(F,F')
=(\ga'_2r+\ga_2r'-\ga_1\ga'_1)
+\oh (\ga_1r'-\ga'_1r)K_S
+\hi(O_S)rr'.}
This implies
\eq{\label{skew-sim}
\ang{\ga,\ga'}:=\hi(\ga,\ga')-\hi(\ga',\ga)
=(\ga_{1}r'-\ga'_{1}r)K_S,}
\eq{\hi(\ga,\ga)=2r\ga_2-\ga_1^2+\hi(O_S)r^2}
Note that if $(r,\ga_1)$ and $(r',\ga'_1)$ are proportional, then $\ang{\ga,\ga'}=0$.

\subsection{Semistability}
\label{sec:semist}
Let $J$ be a polarizing line bundle on $S$.
We denote its first Chern class also by $J$.
We let furthermore
\eq{\mu(F)=\frac{\ga_1}{r},\qquad
\mu_J(F)=\mu(F)\cdot J=\frac{\ga_1\cdot J}{r}}
and let the reduced Hilbert polynomial be $p_J(F,n)=\chi(F \otimes J^n)/r(F)$, or in terms of the Chern character

\eql{
p_J(\ga,n)
=\frac{J^2}2n^2
+\bigg(\frac{\ga_1\cdot J}r-\oh K_S\cdot J\bigg)n
+\bigg(\frac{\ga_2-\oh K_S\cdot\ga_1}r+\chi(\mathcal{O}_S)\bigg). 
}{reduced Hilb}

\begin{remark}\label{rem:generic}
In the following, we let $J$ be any element of the ample cone $C(S)\sbs H^2(S,\mathbb{R})$ rather then of $C(S)\cap H^2(S,\mathbb{Z})$. 
Note that if $\ga$ and $\ga'$ are proportional,
then $p_J(\ga)=p_J(\ga')$.
Conversely, assume that $J$ is generic in the sense that $J\cdot\ga_1=J\cdot\ga'_1$ implies $\ga_1=\ga'_1$ for elements in $H^2(S,\bZ)$.
Equivalently, $J\cdot\ga_1=0$ implies $\ga_1=0$ for $\ga_1\in H^2(S,\bZ)$.
If $\mu_J(\ga)=\mu_J(\ga')$, then $(r,\ga_1)$ and $(r',\ga'_1)$ are proportional.
We conclude from \eqref{skew-sim} that $\ang{\ga,\ga'}=0$ in this case.
If $p_J(\ga)=p_J(\ga')$, then $\ga_1/r=\ga'_1/r'$ and this implies that $\ga_2/r=\ga'_2/r'$. We conclude that $\ga$ and $\ga'$ are proportional.
Note that for any $F\ne0$, we have $\hi(F\ts J^n)>0$ for $n\gg0$.
Therefore, for $\ch(F)=(r,\ga_1,\ga_2)$, we have either $r>0$ or $r=0$ and $\ga_1\cdot J>0$ or $r=\ga_1=0$ and $\ga_2>0$.
\end{remark}

We recall that a sheaf $F$ is called Mumford (or $\mu$-) semi-stable with respect to the
polarization $J$ if for each subsheaf
$F'\subseteq F$,
\be
\label{mustab}
\mu_J(F')\leq \mu_J(F).
\ee
Similarly, a sheaf $F$ is called Gieseker semi-stable with respect to the
polarization $J$ if for each subsheaf
$F'\subseteq F$,
\be
\label{Gstab}
p_J(F',n) \preceq p_J(F,n),
\ee
where $\preceq$ indicates the lexicographic ordering with respect to
the monomials in $n$.

We recall that a Harder-Narasimhan filtration with
respect to a stability condition $\varphi$ is a filtration $0\subset
F_1\subset F_2\subset \dots \subset F_\ell=F$ of a sheaf $F$, such
that the quotients $E_i=F_i/F_{i-1}$ are semi-stable with respect to
$\varphi$ and satisfy $\varphi(E_i)>\varphi(E_{i+1})$ for all $i$. 

Note that for $J=\pm K_S$, the sheaves with equal slopes $\mu_J(\ga)=\mu_J(\ga')$ (or reduced Hilbert polynomials),
have vanishing $\ang{\ga,\ga'}$.

\subsection{Discriminant}
We define the discriminant (cf.\ \cite[\S3.4]{huybrechts_geometry})
\eq{\label{De}
\De(\ga)=\De(F)
=\frac{1}{r}\left( c_2-\frac{r-1}{2r} c_1^2\right)
=\frac{\ga_1^2}{2r^2}-\frac{\ga_2}r.}
Then
\eq{\frac{\hi(\ga,\ga)}{r^2} 
=\frac{2\ga_2}r-\frac{\ga_1^2}{r^2}+\hi(O_S)
=-2\De(\ga)+\hi(O_S).
}
Note that
\eq{\log(\ga/r)
=\frac{\ga_1+\ga_2}{r}-\oh\rbr{\frac{\ga_1+\ga_2}{r}}^2
=\frac{\ga_1}r+\rbr{\frac{\ga_2}r-\frac{\ga_1^2}{2r^2}}
=\mu(\ga)-\De(\ga).
}
Therefore
\eq{\De(\ga\cdot\ga')=\De(\ga)+\De(\ga').}
Note that for a line bundle $L$, we have $r=1$ and $\ga_2=\ga_1^2/2$.
Therefore 
\eq{\De(L)=0,\qquad
\De(F\ts L)=\De(F).}

By the Bogomolov inequality \cite[\S3.4]{huybrechts_geometry}, if $F$ is (Gieseker or Mumford) semistable, then 
\eq{\De(F)\ge0.}

Consider a filtration $0\subset F_1\subset \dots \subset F_\ell=F$ of $F$ whose quotients $E_i=F_i/F_{i-1}$ have Chern character $\gamma(E_i)=\gamma^{(i)}$. Then the discriminant 
$\Delta(\gamma)$ is expressed in terms of $\gamma^{(i)}$ as
\be
\label{Delta_filtration}
r\Delta=\sum_{i=1}^\ell r^{(i)}\Delta^{(i)}-\sum_{i=2}^\ell\frac{1}{2r^{(i)}}\frac{1}{\sum_{j=1}^i r^{(j)} \sum_{k=1}^{i-1}r^{(k)}}\left( \sum_{j=1}^{i-1}r^{(i)}\gamma_{1}^{(j)}-r^{(j)}\ga_{1}^{(i)}\right)^2.
\ee

\subsection{\tpdf{\la}{Lambda}-rings}
For the definition and basic properties of \la-rings see e.g.~\cite{getzler_mixed,mozgovoy_computational}.
Assuming that $R$ is a commutative algebra over \bQ, a \la-ring structure on $R$ is given by a family of ring homomorphisms $(\psi_n:R\to R)_{n\ge1}$, called Adams operations, satisfying
\begin{enumerate}
\item $\psi_1=\Id_R$,
\item $\psi_{m}\psi_n=\psi_{mn}$.
\end{enumerate}
Using the ring of symmetric functions
$$\La=\varprojlim \bZ[x_1,\dots,x_n]^{S_n},$$
we can define a (unique) map $\circ:\La\xx R\to R$, called a plethystic operation, such that
\begin{enumerate}
\item $-\circ r:\La\to R$ is a ring homomorphism for all $r\in R$,
\item $p_n\circ r=\psi_n(r)$ for all $r\in R$ and $n\ge1$, where $p_n=\sum_i x_i^n\in\La$, called power sums.
\end{enumerate}
Assuming that the first of the above axioms is satisfied, the plethystic operation in its turn is uniquely determined by any of the following families of maps
\begin{enumerate}
\item $\la_n(r)=e_n\circ r$, where $e_n=\sum_{i_1<\dots<i_n}x_{i_1}\dots x_{i_n}\in\La$, elementary symmetric functions.
\item $\si_n(r)=h_n\circ r$, where $h_n=\sum_{i_1\le\dots\le i_n}x_{i_1}\dots x_{i_n}\in\La$, complete symmetric functions.
\end{enumerate}
We will see later several examples of \la-rings.
The key example arises from the Grothendieck ring $K_0(\cA)$ of an abelian (or exact) symmetric monoidal category \cA.
The \la- and \si-operations are defined in this case by taking exterior and symmetric powers, respectively
\eq{\la_n(V)=[\La^nV],\qquad \si_n(V)=[S^nV].}
For example, if $\cA=\Vect^\bN$, the category of finite-dimensional \bN-graded vector spaces over a field~\bk, then
$$K_0(\cA)\iso\bZ[x],\qquad [\bk_i]\mto x^i,\ i\ge0,$$
where $\bk_i$ is a one-dimensional vector space concentrated in degree $i$.
One can show that 
$$\si_n(x^i)=x^{ni},\qquad \psi_n(x^i)=x^{ni}$$
and generally
$$\psi_n(f(x))=f(x^n).$$
Generalizing this example further, we can equip the rings
$$\bQ[x_1,\dots,x_k],\qquad
\bQ(x_1,\dots,x_k),\qquad \bQ\pser{x_1,\dots,x_k}$$
with a \la-ring structure
\eq{\psi_n(f(x_1,\dots,x_k))=f(x_1^n,\dots,x_k^n).}
More generally, given a \la-ring $R$ and a commutative monoid \Ga, we can equip the semigroup algebra
\eq{R[\Ga]=\bop_{\ga\in\Ga}R=\sets{f=\sum_{\text{finite}} f_\ga z^\ga}{f_\ga\in R\ \forall\ga\in\Ga},}
with a (\Ga-graded) \la-ring structure
\eq{\psi_n\rbr{\sum f_\ga z^\ga}=\sum\psi_n(f_\ga)z^{n\ga}.}
Assuming that \Ga is locally finite, that is,
\eql{\#\sets{(a,b)\in\Ga\xx\Ga}{a+b=c}<\infty
\qquad\forall c\in\Ga,
}{loc fin}
we can also equip the completion
\eq{R\pser\Ga=\prod_{\ga\in\Ga}R}
with a \la-ring structure.
We define the plethystic exponential on $R\pser\Ga$
\eql{\Exp(f)=\sum_{n\ge0}\si_n(f)=\exp\rbr{\sum_{n\ge1}\frac1n\psi_n(f)},\qquad f=\sum f_\ga z^\ga,\quad f_0=0.}{eq:Exp}
Its inverse, the plethystic logarithm, is given by
\eql{\Log(f)=\sum_{n\ge1}\frac{\mu(n)}n\psi_n(\log(f)),\qquad f=\sum f_\ga z^\ga,\quad f_0=1,}{eq:Log}
where $\mu$ is the classical M\"obius function.

\subsection{Some Grothendieck groups}
Given a scheme (or an algebraic stack) $S$ locally of finite type over \bC, let $K_0(\var S)$ denote the free abelian group generated by isomorphism classes of objects in $\var S$ (the category of finite type schemes over $S$), modulo relations
\eq{[X\to S]=[Z\to S]+[U\to S],}
where $Z\sbs X$ is a closed subvariety and $U=X\backslash Z$ is its complement.
Sometimes we denote $K_0(\var\pt)$ (where $\pt=\Spec\bC$) by $K_0(\var\bC)$.
It has a structure of a ring and the group $K_0(\var S)$ has a module structure over it.
The element $\bL=[\bA^1]\in K_0(\var\bC)$ is called the Lefschetz motive.
One can show that by localizing $K_0(\var S)$ with respect to $\bL$ and $\bL^n-1$ for $n\ge1$, one obtains the Grothendieck group $K_0(\St S)$ of (finite type) stacks with affine stabilizers over $S$ \cite[\S3.3]{bridgeland_introduction}.
Let $\kvar\pt$ be obtained from $K_0(\var\pt)\ts\bQ$ by localizing with respect to the above elements and by adjoining the element $\bL^\oh$.
Generally, define
\eq{\kvar S
=K_0(\var S)\ts_{K_0(\var\pt)}\kvar\pt.}
If $X\to S$ is a finite type stack with affine stabilizers over $S$, we denote by $[X\to S]$ the corresponding element in $\kvar S$.
The rings $K_0(\var\pt)$ and $\kvar\pt$ are equipped with (pre-)\la-ring structures
\eq{\si^n(X)=[S^nX].}
The Adams operations act on \bL as $\psi_n(\bL)=\bL^n$.
The action on $\bL^\oh$ is defined to be $\psi_n(-\bL^\oh)=(-\bL^\oh)^n$.

For any quasi-projective variety $S$, the Grothendieck group $K_0(\MHM(S))$ of mixed Hodge modules over $S$ is a module over the ring $K_0(\MHS)=K_0(\MHM(\pt))$.
Similarly to the construction of $\kvar S$, we consider \eq{\bL=H^*_c(\bA^1,\bQ)=\bQ(-1)[-2]}
as an element of $K_0(D^b\MHS)=K_0(\MHS)$ and 
define the ring $\kmhm\pt$ by localizing $K_0(\MHS)\ts\bQ$ with respect to $\bL$ and $\bL^n-1$, $n\ge1$ and by adjoining the element $\bL^\oh$.
Then we define
\eq{\kmhm S=K_0(\MHM (S))\ts_{K_0(\MHM(\pt))}\kmhm\pt.}

There is a well-defined group homomorphism
\eq{\hi_c:K_0(\var S)\to K_0(\MHM(S)),\qquad [f:X\to S]\mto [f_!\bQ_X]}
which extends to
\eq{\hi_c:\kvar S\to \kmhm S.}

\begin{remark}
To see that the map is indeed well-defined, consider $f:X\to S$, a closed embedding $i:Z\to X$ and its open complement $j:U\to X$.
Then there is a distinguished triangle
$$j_!j^!F\to F\to i_*i^*F\to$$
for any $F\in D^b(\MHM(X))$.
Considering $F=\bQ_X$ and applying $f_!$ for $f:X\to S$, we obtain a distinguished triangle
$$(fj)_!\bQ_U\to f_!\bQ_X\to (fi)_!\bQ_Z\to$$
which implies
$$\hi_c(f)=\hi_c(fi)+\hi_c(fj).$$
\end{remark}

In particular, we have a map
\eq{\hi_c:K_0(\var\bC)\to K_0(\MHS),
\qquad [X]\mto H^*_c(X,\bQ)
}
which was proved to be a homomorphism of (pre-)\la-rings
in \cite[\S2.2]{maxim_twisted}.

\subsection{\tpdf{$E$}{E}-polynomial and Poincar\'e polynomial}
\label{sec:E-pol}
Given a mixed Hodge structure $V$, we define its Hodge-Euler polynomial \cite[\S3.1]{peters_mixed}
\eql{E(V,u,v)=\sum_{p,q}h^{p,q}(V)u^pv^q,
\qquad h^{p,q}(V)=\dim\Gr^p_F\Gr_{p+q}^W(V_\bC).
}{E-pol}
We can extend $E$ to a \la-ring homomorphism
$$E:K_0(\MHM(\pt))=K_0(\MHS)\to\bQ[u^{\pm1},v^{\pm1}],$$
where the \la-ring structure on the right is given by $\psi_n(f(u,v))=f(u^n,v^n)$.
We can also extend $E$ to a \la-ring homomorphism
$$E:\kmhm\pt\to\bQ(u^{\oh},v^{\oh})$$
with $E(\bL^\oh)=-(uv)^\oh$.
We will also denote by $E$ the composition
\eq{\kvar\pt\xto{\hi_c}\kmhm\pt\xto E\bQ(u^{\oh},v^{\oh})}
called the $E$-polynomial (or Hodge-Deligne polynomial), although it is a rational function in general.
For an algebraic variety $X$, we have
\eq{E(X)=\sum_n(-1)^n\sum_{p,q}h^{p,q}(H^n_c(X,\bQ))u^pv^q.}
We define the (motivic) Poincar\'e polynomial $$P:\kvar\pt\to\bQ(y),\qquad
P(X)=P(X,y)=E(X,y,y).$$
If $X$ is a smooth projective variety then
\eql{P(X)=\sum_{n\ge0}\dim H^n(X,\bC)(-y)^n.}{eq:poincare proj}
Moreover, $P(\bL^\oh)=-y$.

\subsection{Virtual intersection complexes and motives}
Given an algebraic variety $X$ of dimension~$d$, define the virtual intersection complex
\eq{\ICV_X=\IC_X(d/2)=\bL^{-\oh d}\IC_X[-d].}
This is a weight zero Hodge module.
Given a map $f:X\to S$, where $X$ is an algebraic variety (or a finite type stack with affine stabilizers over $S$) of dimension $d$, define
\eq{[X\to S]_\vir=\bL^{-\oh d}[X\to S].}
If $X$ is smooth then
\eql{\hi_c([X\to S]_\vir)
=\bL^{-\oh d}f_!(\bQ_X)=f_!(\ICV_X)}{vir mot to mhm}
as
$\IC_X=\bQ_X[d]$ and $\ICV_X=\bL^{-\oh d}\bQ_X$.
In particular, if $X$ is an algebraic variety (or stack) of dimension $d$, then
\eq{[X]_\vir=\bL^{-\oh d}[X]\in\kvar\pt}
and if $X$ is smooth, then
\eq{\hi_c([X]_\vir)=H^*_c(X,\ICV_X).}

\subsection{Graded commutative monoids}
We will construct generalizations of the rings $\kvar\pt$ and $\kmhm\pt$.
Let $\Ga\sbs\bZ^n$ be a monoid
and let $M=\bigsqcup_{\ga\in\Ga}M_\ga$ be a \Ga-graded commutative monoid in the category of complex algebraic varieties.
This means that $M_\ga$ are complex algebraic varieties (we will assume that they are quasi-projective) equipped with an associative commutative multiplication
$$\mu:M_\ga\xx M_{\ga'}\to M_{\ga+\ga'}$$
and with a unit $\pt\to M_0$ satisfying the standard axioms.
We will assume that the map $\mu$ is finite.
Define a \Ga-graded group
\eq{\kvar M=\bop_{\ga\in\Ga}\kvar{M_\ga}}
and equip it with a commutative ring structure
\eq{[X\to M_\ga]\cdot[Y\to M_{\ga'}]
=[X\xx Y\to M_\ga\xx M_{\ga'}\xto\mu M_{\ga+\ga'}].
}
It has a (pre-)\la-ring structure defined by
\eq{\si_n(X\to M_\ga)=[S^nX\to S^nM_\ga\xto\mu M_{n\ga}].}

On the other hand, consider a \Ga-graded category
$$\cA=\bigsqcup_{\ga\in\Ga}\cA_\ga,\qquad \cA_\ga=\MHM(M_\ga)$$
and equip it with the tensor product
\eq{\odot:\cA_\ga\xx \cA_{\ga'}\to \cA_{\ga+\ga'},\qquad E\odot F=\mu_*(E\boxtimes F),}
where $E\boxtimes F=p_1E\ts p_2F$ with $p_1:M_\ga\xx M_{\ga'}\to M_{\ga}$, $p_2:M_\ga\xx M_{\ga'}\to M_{\ga'}$ being projections.
It is proved in \cite{mozgovoy_intersection} that $\cA$ equipped with this tensor product is a symmetric monoidal category.

The Grothendieck groups
\eq{K_0(\cA)=\bop_{\ga\in\Ga}K_0(\cA_\ga),\qquad 
\kmhm M=\bop_{\ga\in\Ga}\kmhm{M_\ga}}
are commutative \Ga-graded rings with multiplication
\eq{[E]\cdot[F]=[E\odot F].}
By \cite{mozgovoy_intersection,biglari_rings,maxim_twisted,getzler_mixed,heinloth_note} they are also \la-rings with \si-operations defined by
\eql{\si_n(E)=S^nE=\im\rbr{\frac1{n!}\sum_{\ta\in S_n}\ta}\sbs E^{\ts n}.}{eq:sigma}

One can prove that the map
\eql{\hi_c:\kvar M\to \kmhm M,\qquad [f:X\to M_\ga]\mto f_!\bQ_X}{eq:hic on M}
is a homomorphism of (pre-)\la-rings using results of \cite{maxim_symmetric}.
If \Ga is locally finite \eqref{loc fin}, then we can equip
\eql{\hkvar{\mm{}}=\prod_\ga \kvar{\mm\ga},\qquad
\hkmhm{\mm{}}=\prod_\ga \kmhm{\mm\ga}}{complete Kvar}
with the structures of (pre-)\la-rings and extend \eqref{eq:hic on M} to a homomorphism of (pre-)\la-rings
\eq{\hi_c:\hkvar M\to \hkmhm M.}

\section{The main result}
\label{result}
Let $S$ be a projective surface over \bC and $J\in H^2(S,\bR)$ be a polarization on $S$.
We will assume that $J$ is generic and $J\cdot K_S<0$.
The latter requirement is needed because of the following result

\begin{lemma}
\label{lm:hereditary}
Assume that $J\cdot K_S<0$.
Then for any Gieseker (or Mumford) semistable sheaves $E,F\in\Coh S$ with $p_J(E)=p_J(F)$ (or $\mu_J(E)=\mu_J(F)$), we have
$$\Ext^2(E,F)=0.$$ 
\end{lemma}
\begin{proof}
Gieseker semistability implies Mumford semistability.
Therefore we can assume that $E,F$ are Mumford semistable
and $\mu_J(E)=\mu_J(F)$.
Then the sheaf $E\ts K_S$ is also Mumford semistable and $\mu_J(E\ts K_S)<\mu_J(E)=\mu_J(F)$.
This implies $\Hom(F,E\ts K_S)=0$.
By the Serre duality
$$\Ext^2(E,F)\iso\Hom(F,E\ts K_S)^*=0.$$
\end{proof}

Given a polynomial $p$, let $\Ga^*\sbs \oh H^\even(S,\bZ)$ be a semigroup consisting of classes $\ga=(r,\ga_1,\ga_2)$ with $p_J(\ga)=p$ and $r>0$ and let $\Ga=\Ga^*\cup\set0$.
If $J$ is generic then \Ga is isomorphic to \bN.
For any $\ga\in \Ga$, let $\mm \ga$ (resp.\ $\ms\ga$) denote the moduli space (resp.\ stack) of Gieseker semi-stable sheaves on $S$ having Chern character \ga.
We let $\mm 0=\pt$.
The schemes $\mm\ga$ are projective, possibly singular \cite[Theorem 4.3.4]{huybrechts_geometry}.
The goal of this section is to prove the following (cf.\ Theorem \ref{thm2})

\begin{theorem}
If $J$ is generic and $J\cdot K_S<0$ then
$$\sum_{\ga\in\Ga^*}\bL^{-\oh\dim\mm\ga}\IE(\mm\ga)z^\ga
=(\bL^\oh-\bL^{-\oh})\Log\rbr{\sum_{\ga\in\Ga}\bL^{-\oh\dim\ms\ga}E(\ms\ga)z^\ga}$$
in $\bQ(u^{\oh},v^{\oh})\pser{\Ga}$ with $\bL=E(H^*_c(\bA^1,\bQ))=uv$.
\end{theorem}
This theorem relates the $E$-polynomials of the intersection cohomologies $\IH^*(\mm\ga,\bQ)$ and the $E$-polynomials of the stacks $\ms\ga$ (or of the elements $\hi_c([\ms\ga])\in\kmhm\pt$).
We can formulate it on the level of mixed Hodge structures.
\begin{theorem}
If $J$ is generic and $J\cdot K_S<0$ then
$$\sum_{\ga\in\Ga^*}\bL^{-\oh\dim\mm\ga}\IH^*(\mm\ga)z^\ga
=(\bL^\oh-\bL^{-\oh})\Log\rbr{\sum_{\ga\in\Ga}\bL^{-\oh\dim\ms\ga}\hi_c([\ms\ga])z^\ga}$$
in $\kmhm{\pt}\pser\Ga$ with $\bL=[H^*_c(\bA^1,\bQ)]=[\bQ(-1)[-2]]$.
\end{theorem}
This statement can be further generalized to an equation in the \la-ring $\hkmhm{\mm{}}$ \eqref{complete Kvar}, where $M=\bigsqcup_{\ga\in\Ga}\mm\ga$ is a \Ga-graded commutative monoid with a finite multiplication
\eq{\mu:\mm\ga\xx\mm{\ga'}\to\mm{\ga+\ga'},\qquad (E,F)\mto E\oplus F.}

\begin{remark}
We can consider \Ga as a \Ga-graded monoid in the category of algebraic varieties, which consists of a single point at every degree.
Then $\hkmhm{\Ga}\iso\kmhm\pt\pser\Ga$.
The natural projection $a:\mm{}\to\Ga$, $\mm\ga\to\set\ga$ is a homomorphism of \Ga-graded monoids and induces a \la-ring homomorphism
\eq{a_!:\hkmhm{\mm{}}\to\kmhm\pt\pser\Ga,\qquad [F\in\MHM(M_\ga)]\mto H^*_c(M_\ga,F)z^\ga.}
We note that 
$$a_!(\ICV_{\mm\ga})=\bL^{-\oh\dim\mm\ga}\IH^*(\mm\ga)z^\ga,\qquad
a_!\hi_c([\ms\ga\to\mm\ga])=\hi_c([\ms\ga])z^\ga.
$$
\end{remark}

\begin{theorem}\label{th:main3}
If $J$ is generic and $J\cdot K_S<0$ then
$$\sum_{\ga\in\Ga^*}\ICV_{\mm\ga}
=(\bL^\oh-\bL^{-\oh})
\Log\rbr{\sum_{\ga\in\Ga}\hi_c\rbr{[\ms\ga\to\mm\ga]_\vir}}$$
in $\hkmhm{\mm{}}$ with $\bL=[H^*_c(\bA^1,\bQ)]=[\bQ(-1)[-2]]$.
\end{theorem}

Our goal will be to prove this theorem.
The strategy of the proof follows the ideas of \cite{meinhardt_donaldson,mozgovoy_intersection,meinhardt_donaldsona}. Namely, we construct auxiliary smooth moduli spaces $\fm\ga$ of framed sheaves together with projections $\pi:\fm\ga\to\mm\ga$.
Then we relate $\pi_!\ICV_{\fm\ga}$ to both sides of the above theorem.
In contrast to \cite{meinhardt_donaldson,mozgovoy_intersection,meinhardt_donaldsona}, we will need to overcome a technical difficulty arising from the fact that the framing functors (see \S\ref{sec:framed}) needed in the construction of $\fm\ga$ are not exact.

\subsection{Motivic Hall algebra and DT invariants}
Let
$$\mm{}=\bigsqcup_{\ga\in\Ga}\mm\ga,\qquad
\ms{}=\bigsqcup_{\ga\in\Ga}\ms\ga.$$
There are natural maps $p_\ga:\ms\ga\to \mm\ga$ and $p:\ms{}\to \mm{}$.
We have seen that $\mm{}$ is a commutative \Ga-graded monoid.
Therefore $\kvar{\mm{}}$ is equipped with a (pre-)\la-ring structure.
We define the motivic Hall algebra
$$H=\kvar{\ms{}}=\bop_{\ga\in\Ga}H_\ga,\qquad H_\ga=\kvar{\ms\ga}$$
with the Ringel-Hall multiplication as in \cite[\S4.2]{bridgeland_introduction} (see also \cite{joyce_configurationsa}).
The following map, called an integration map,
\eq{I:H\to \kvar{\mm{}},\qquad
[X\to \ms \ga]\mto \bL^{-\oh\dim\ms \ga}[X\to\ms \ga\to \mm \ga]}
is an algebra homomorphism by \cite{joyce_configurationsa,reineke_counting} if the category of semistable sheaves with Chern characters $\ga\in\Ga$ has homological dimension one (this is the case under our assumptions by Lemma \ref{lm:hereditary}).
Note that $\dim\ms \ga=-\hi(\ga,\ga)$.

\begin{remark}
To make $I$ an algebra homomorphism, one actually defines multiplication in $\kvar{\mm{}}$ as
$$[X\to \mm \ga]\cdot [Y\to \mm {\ga'}]
=\bL^{\oh(\hi(\ga,\ga')-\hi(\ga',\ga))}[X\xx Y\to \mm {\ga+\ga'}].$$
But $\hi(\ga,\ga')=\hi(\ga',\ga)$ if $p_J(\ga)=p_J(\ga')$ (or $\mu_J(\ga)=\mu_J(\ga')$) and $J$ is generic, hence we omit the twist.
This is also true if $J=\pm K_S$ and all our arguments will also work in this situation.
\end{remark}

Previously we defined ring homomorphisms
$$H\xto I \kvar{\mm{}}\xto{\hi_c}\kmhm{\mm{}}.$$
Similarly, we can define completions
$$\what H=\prod_\ga H_\ga,\qquad 
\hkvar{\mm{}}=\prod_\ga \kvar{\mm\ga},\qquad
\hkmhm{\mm{}}=\prod_\ga \kmhm{\mm\ga}.$$
and ring homomorphisms between them.

We define motivic Donaldson-Thomas invariants $$\DT^\mot_{\ga}\in\kvar{\mm\ga}$$ 
by the formula
\eql{
I\rbr{\one_{\ms{}}}
=\sum_{\ga\in\Ga}\bL^{-\oh\dim\ms\ga}[\ms\ga\to \mm\ga]
=\Exp\rbr{\frac{\sum_{\ga\in\Ga^*}\DT^\mot_\ga}{\bL^\oh-\bL^{-\oh}}}
,}{eq:mot DT}
where $\one_{\ms{}}:\ms{}\to\ms{}$ is an identity map.
Equivalently,
\eql{\sum_{\ga\in\Ga^*}\DT^\mot_\ga
=\rbr{\bL^\oh-\bL^{-\oh}}
\Log\rbr{\sum_{\ga\in\Ga}[\ms\ga\to \mm\ga]_\vir}.
}{eq:DT mot}

We define MHM Donaldson-Thomas invariants as
$$\DT_\ga^\mhm=\hi_c(\DT_\ga^\mot)\in\kmhm{\mm\ga}$$
or, equivalently, by the formula
\eql{\hi_c\, I\rbr{\sum_{\ga}\one_{\ms\ga}}
=\Exp\rbr{\frac{\sum_{\ga\ne0}\DT^\mhm_\ga}{\bL^\oh-\bL^{-\oh}}}
.}{eq:mhm DT}

\subsection{Framed moduli spaces}
\label{sec:framed}
Let $\cA=\Coh S$ and $T\in\cA$ be some object.
We consider a left exact functor
$$\Phi:\cA\to\Vect,\qquad \Phi(E)=\Hom(T,E)$$
and define the category $\cA_{f}$ of framed objects as follows.
Its objects are triples $(E,V,s)$, where $E\in\cA$, $V\in\Vect$ and $s:V\to\Phi(E)$ is a linear map.
A morphism $f:(E,V,s)\to(E',V',s')$ is a pair $(f_1,f_2)$, where $f_1:E\to E'$ and $f_2:V\to V'$ satisfy $\Phi(f_1)s=s'f_2$.
One can show that $\cA_f$ is an abelian category.
We will denote an object $(E,0,0)$ by $E$ and call it an unframed object.
We will denote an object $(E,\bC,s)$ by $(E,s)$ and consider $s$ as an element of $\Phi(E)$.

Given a coherent sheaf $E$ with Chern character~$\ga$, we define
\eq{\vi(\ga)=\vi(E):=\sum_{i\ge0}(-1)^i\dim R^i\Phi(E)=\hi(T,E).}
In order to construct moduli spaces of stable framed objects we will use results of \cite{huybrechts_stable} (or more precisely, the dual version of these results).
First, we will reformulate the stability condition from~\cite{huybrechts_stable}.
Given a polynomial $\de\in\bQ[n]$, define for any triple $\lb E=(E,V,s)$
$$p_\de(\lb E)=\frac{P_J(E)+\dim V\cdot \de}{\rk(E)},$$
where $P_J(E)$ is the Hilbert polynomial of $E$ (with respect to the polarization $J$).
We will say that an object $\lb E$ is $\de-$stable if for any proper $\lb G\sbs\lb E$ we have
$$p_\de(\lb G)<p_\de(\lb E).$$
In the case of a pair $\lb E=(E,s)$ this means
\begin{enumerate}
\item For any $G\sbs E$, we have $p_J(G)<p_\de(\lb E)$.
\item For any proper $G\sbs E$ with $s\in\Phi(G)$, we have
$p_\de(\lb E)<p_J(E/G)$.
\end{enumerate}
Here $p_J(E)$ is the reduced Hilbert polynomial of $E$.
In \cite{huybrechts_stable} the authors constructed the moduli spaces of \de-stable objects.
For our applications we will consider \de to be a constant such that $0<\de\ll1$.
Then $(E,s)$ is stable if and only if
\begin{enumerate}
\item For any $G\sbs E$, we have $p_J(G)\le p_J(E)$. Equivalently, $E$ is Gieseker semistable.
\item For any proper $G\sbs E$ with $s\in\Phi(G)$, we have
$p_J(E)<p_J(E/G)$. Equivalently, $p_J(G)<p_J(E)$.
\end{enumerate}

Let $\fm\ga$ denote the moduli space of stable pairs $(E,s)$ with $\ch E=\ga$. This is a projective variety and there is a projective map
$$\pi:\fm\ga\to \mm\ga,\qquad (E,s)\mto E,$$
where $\mm\ga$ denotes as before the moduli space of (Gieseker) semistable sheaves on $S$ with Chern character $\ga$.

These moduli spaces are instrumental for a new description of the DT invariants.
Let $\cB\sbs\cA=\Coh S$ be an abelian category of Gieseker semistable sheaves $E$ with the reduced Hilbert polynomial $p_J(E)=p$.
Let $\nn-$ be some norm on $H^\even(S,\bR)$.

\begin{definition}
Given a constant $N>0$ and a left exact functor $\Phi:\cA\to\Vect$, we will say that \Phi is $N$-exact if $R^i\Phi(E)=0$ for $i>0$ and $\Phi(E)\ne0$
for all semistable $E$ with $\nn{\ch E}\le N$.
\end{definition}

\begin{remark}
\label{good Phi}
For a fixed $N>0$, the set of $\ga\in\Ga$ with $\nn\ga\le N$ is finite, hence the family of semistable sheaves of type $\ga$ with $\nn\ga\le N$ is bounded \cite[Theorem 3.3.7]{huybrechts_geometry} in the sense of \cite[\S1.7]{huybrechts_geometry}.
This implies that we can choose $T=L^{-n}$, where $L$ is an ample line bundle and $n\gg0$ such that
\begin{enumerate}
\item $\Ext^i(T,E)=0$, $i>0$
\item $\Hom(T,E)\ne0$
\end{enumerate}
for all semistable $E$ with Chern character $\ga$ and $\nn\ga\le N$.
Therefore $\Phi=\Hom(T,-)$ is $N$-exact.
This rather arbitrary choice of $T$ indicates that the moduli spaces $\fm\ga$ play a purely auxiliary role in our analysis of the moduli spaces $\mm\ga$.
\end{remark}


\begin{remark}
Let us show that if $\Phi$ is $N$-exact then $\fm\ga$ are smooth for $\nn\ga\le N$ as we will need this fact when we will work with intersection complexes on $\fm\ga$.
If $E$ is semistable and has Chern character \ga, then $\dim\Phi(E)=\vi(\ga)$ (under our assumptions on $\Phi$ and $\ga$).
The moduli stack $\fs\ga$ of semistable framed objects is open in $\ms\ga\xx \bA^{\vi(\ga)}$, where $\ms\ga$ is a smooth moduli stack of semistable sheaves.
As the automorphism groups of objects in $\fs\ga$ are trivial, we conclude that $\fm\ga$ is smooth.
\end{remark}

\begin{theorem}
\label{framed and DT}
We have
$$\sum_{p_J(\ga)=p}(-1)^{\vi(\ga)}[\fm\ga\to \mm \ga]_\vir
=\Exp\rbr{\sum_{p_J(\ga)=p,\ga\ne0}(-1)^{\vi(\ga)}[\bP^{\vi(\ga)-1}]_\vir\DT^\mot_\ga}$$
in $\kvar{\mm{}}$,
for summands with $\nn\ga\le N$.
\end{theorem}
\begin{proof}
Let $\cB\sbs\cA$ be the subcategory of Gieseker semistable vector bundles $E$ with $p(E)=p$.
One can show that for a pair $(E,s)$ with $E\in\cB$ there exists a unique stable subobject
$$(E',s)\sbs (E,s)$$
with $E',E/E'\in\cB$ (this is a Harder-Narasimhan filtration with respect to an appropriate stability condition on $\cB_\fr$).
Let 
$$\one^{\fr,s}_\cB=\sum_{\ga\in\Ga}
[\fs\ga\to\ms\ga]\in\hat H$$
and similarly let $\one^\fr_\cB\in\hat H$ parametrize all pairs $(E,s)$ with $E\in\cB$.
Let
$$\one_\cB=\sum_{\ga\in\Ga}
[\ms\ga\to\ms\ga]\in\hat H.$$
Then the above Harder-Narasimhan filtration translates to
an equation
$$\one^\fr_\cB=\one^{\fr,s}_\cB\circ\one_\cB$$
in the Hall algebra $\what H$.
We should stress that this is a relation in the Hall algebra of $\cB$, although we used the Harder-Narasimhan filtration in the category $\cB_\fr$.
Applying the integration map $I:\what H\to \hkvar{\mm{}}$, we obtain the following relation
$$\sum_\ga\bL^{\vi(\ga)}[\ms\ga\to\mm\ga]_\vir
=\sum_\ga\bL^{-\oh\dim\ms\ga}[\fm\ga\to\mm\ga]
\cdot\sum_\ga[\ms\ga\to\mm\ga]_\vir
$$
for $\nn\ga\le N$.
Using the formula
$$\dim \fm\ga
=-\hi(\ga,\ga)+\vi(\ga)=\dim\ms\ga+\vi(\ga)$$
we obtain
$$\sum_\ga\bL^{\vi(\ga)}[\ms\ga\to\mm\ga]_\vir
=\sum_\ga\bL^{\oh\vi(\ga)}[\fm\ga\to\mm\ga]_\vir
\cdot\sum_\ga[\ms\ga\to\mm\ga]_\vir.
$$
This can be written in terms of DT invariants
$$
\sum_\ga\bL^{\oh\vi(\ga)}[\fm\ga\to\mm\ga]_\vir
=\Exp\rbr{\sum_\ga(\bL^{\vi(\ga)}-1)\frac{\DT^\mot_\ga}{\bL^\oh-\bL^{-\oh}}}.
$$
Applying the (plethystic) change of variables
$$x\mto (-\bL^\oh)^{-\vi(\ga)}x,\qquad x\in \kvar{\mm\ga},$$
we obtain
\begin{multline*}
\sum_\ga(-1)^{\vi(\ga)}[\fm\ga\to\mm\ga]_\vir
=\Exp\rbr{\sum_\ga(-1)^{\vi(\ga)}\frac{\bL^{\oh\vi(\ga)}-\bL^{-\oh\vi(\ga)}}{\bL^\oh-\bL^{-\oh}}\DT^\mot_\ga}\\
=\Exp\rbr{\sum_\ga(-1)^{\vi(\ga)}[\bP^{\vi(\ga)-1}]_\vir\DT^\mot_\ga}.
\end{multline*}

\end{proof}

\begin{proof}[Proof of Theorem \ref{th:main3}]
By comparing the statement of the theorem and the definition of DT invariants \eqref{eq:DT mot} we have to prove
\eq{\DT^\mhm_\ga=\ICV_{\mm \ga}.}
We can assume by induction that $\DT^\mhm_\al=\ICV_{\mm \al}$ for $\nn\al<N:=\nn\ga$.
Let $T=L^{-n}$, where $L$ is an ample line bundle, and $\Phi=\Hom(T,-)$ be as in Remark~\ref{good Phi}.
Then assumptions of Theorem \ref{framed and DT} are satisfied and we can apply the map $\hi_c:\kvar{\mm{}}\to\kmhm{\mm{}}$ to its statement.
As $\fm\ga$ is smooth, we obtain from \eqref{vir mot to mhm}
$$\hi_c([\fm\ga\to\mm\ga]_\vir)
=\pi_!\ICV_{\fm\ga}.$$
Therefore Theorem \ref{framed and DT} implies
$$(-1)^{\vi(\ga)}\pi_!\ICV_{\fm\ga}
=\sum_{\over{m:\Ga_*\to\bN}{\sum m_\al\al=\ga}}\prod_\al S^{m_\al}\rbr{(-1)^{\vi(\al)}[\bP^{\vi(\al)-1}]_\vir\DT^\mhm_{\al}}.$$
\def\rfa#1{\cite[#1]{mozgovoy_intersection}}
Now we literally repeat the arguments of \cite[Theorem 5.4]{mozgovoy_intersection} to conclude that $\DT_\ga^\mhm=\ICV_{\mm\ga}$.
For all these arguments to work it is enough to assume that $\Phi$ is $N$-exact.
\end{proof}

\section{Some properties of DT invariants}
\label{generating}
As before, we assume that $J$ is a generic (ample) polarization on a surface $S$ with $J\cdot K_S<0$.
Let $\mm\ga$ be the moduli space Gieseker semistable sheaves with Chern character \ga.
Let $\ms\ga=\ms J^\rG(\ga)$ be the moduli stack of Gieseker semistable sheaves
and $\ms\ga^\rM=\ms J^\rM(\ga)$ be the moduli stack of Mumford (or $\mu$-) semistable sheaves with Chern character \ga. 

In the previous section we studied Donaldson-Thomas invariants (\ref{eq:mot DT}, \ref{eq:mhm DT})
$$\DT^\mot_\ga\in\kvar{\mm\ga},\qquad
\DT^\mhm_\ga\in\kmhm{\mm\ga}.
$$
For the actual computations it is more appropriate to study their images in $\kmhm\pt$ or merely their $E$-polynomials or (motivic) Poincar\'e polynomials.
Thus, we define Donaldson-Thomas invariants
\eq{\Om_\ga=P(a_!\DT^\mhm_\ga)=P(a_!\hi_c\DT^\mot_\ga)\in\bQ(y).}
where $a:\mm\ga\to\pt$ is a projection.
Applying \eqref{eq:mot DT}, we can write equivalently
\eq{1+\sum_{p_J(\ga)=p}\cI_\ga z^\ga
=\Exp\rbr{\frac{\sum_{p_J(\ga)=p}\Om_\ga z^\ga}{y\inv-y}},
}
where
\eq{\cI_\ga=\cI(\ga,y;J)=(-y)^{\hi(\ga,\ga)}P(\ms\ga).}
Let us give an explicit formula.
We define the rational invariant 
$\Omb_\ga$ of $\ms\ga$, which is given in terms of $\cI_\ga$ by \cite[Definition 6.22]{Joyce:2004tk}
\eql{\Omb_\ga
=\Omb(\ga,y;J)
=\sum_{\over{\ga_1+\dots+\ga_\ell=\ga}
{p_J(\ga_i)=p_J(\ga)\ \forall i}} \frac{(-1)^{\ell-1}}{\ell}
\prod_{i=1}^\ell \cI(\ga_i,y;J). 
}{I to Omb}
The inverse relation is given by
\eql{
\cI(\ga,y;J)=\sum_{\over{\ga_1+\dots+\ga_\ell=\ga}
{p_J(\ga_i)=p_J(\ga)\ \forall i}}
\frac{1}{\ell!}\prod_{i=1}^\ell \Omb(\ga_i,y;J).
}{Om to I}
Finally, we define the Donaldson-Thomas invariant
\eql{
\Om_\ga=\Om(\ga,y;J)
=(y\inv-y)\sum_{m|\ga}\frac{\mu(m)}m
\Omb(\ga/m,y^m;J),
}{bOmtoOm}
with inverse relation
\eq{
\Omb(\ga,y;J)=\sum_{m|\ga}\frac1m\frac{\Om(\ga/m,y^m;J)}{y^{-m}-y^m}.
}
The main result of the previous section implies
\eq{\Om_\ga=(-y)^{-\dim M_\ga}\sum_n\dim\IH^n(M_\ga)(-y)^n,}
therefore $\Om_\ga\in\bQ[y^{\pm1}]$.

Similarly, we define invariants $\cI^\rM_\ga$, $\Omb^\rM_\ga$, and $\Om^\rM_\ga$ of the moduli stacks $\ms\ga^\rM$.
In particular, for any $\ta\in\bR$,
\eq{1+\sum_{\mu_J(\ga)=\ta}\cI^\rM_\ga z^\ga
=\Exp\rbr{\frac{\sum_{\mu_J(\ga)=\ta}\Om^\rM_\ga z^\ga}{y\inv-y}}.
}

\begin{theorem}
\label{Om G and S}
If $J$ is generic and $J\cdot K_S<0$, then
$$\Om_\ga=\Om_\ga^M.$$
\end{theorem}
\begin{proof}
Every Mumford semistable sheaf $F$ has a unique Harder-Narasimhan filtration 
$$0=F_0\sbs F_1\sbs \dots\sbs F_n=F$$
with respect to the Gieseker stability.
The factors of this filtration are Gieseker (hence Mumford) semistable with slope $\mu_J(F)$.
Let $\ga=(r,\ga_1,\ga_2)=\ch F$ and $\ga^{(i)}=(r^{(i)},\ga_1^{(i)},\ga_2^{(i)})=\ch F_i/F_{i-1}$ for $i=1,\dots,n$.
The sequence
$$(\ga^{(1)},\dots,\ga^{(n)})$$
is called the type of the HN filtration and we claim that there occurs a finite number of such types for the family of all Mumford semistable sheaves of fixed type \ga.
As $J$ is generic, we have $\ga_1^{(i)}/r^{(i)}=\ga_1/r$,
hence the number of possible pairs $(r^{(i)},\ga_1^{(i)})$ is finite.
We conclude by the Bogomolov inequality ($\De(\ch F)\ge0$ for a Mumford semistable sheaf $F$) that $\ga_2^{(i)}$ are bounded above and therefore there is a finite number of possible classes $\ga^{(i)}$ appearing in the HN filtrations.

For a fixed $\ta\in\bR$, let $\Ga^*$ be the set of all Chern characters $\ga=(r,\ga_1,\ga_2)$ with $\mu_J(\ga)=\ta$ and $\De(\ga)\ge0$.
Let $\Ga=\Ga^*\cup\set0$.
Then $\ga_1/r$ is independent of $\ga\in\Ga^*$ and therefore $\ga_2/r\le\ga_1^2/2r^2=:\nu_\ta$ is bounded above.
This implies that $\Ga$ is a locally finite monoid \eqref{loc fin}.
Using the formula for the reduced Hilbert polynomial \eqref{reduced Hilb}, we can write for any $\ga\in\Ga$
\eql{
p_J(\ga,n)=p_{J,\ta}(n)+\frac{\ga_2}r,}{eq:p J ta}
where the polynomial $p_{J,\ta}$ is independent of $\ga\in\Ga$.
This implies that for $\ga,\ga'\in\Ga$
$$p_J(\ga)\le p_J(\ga')\quad\iff\quad \ga_2/r\le\ga'_2/r'.$$
For $\nu\in\bR$, let 
$$\Ga^*_\nu=\sets{\ga\in\Ga^*}{p_J(\ga,n)=p_{J,\ta}+\nu},
\qquad \Ga_\nu=\Ga_\nu^*\cup\set0.$$

Uniqueness of the Harder-Narasimhan filtration implies a relation in the motivic Hall algebra (of the category of Mumford semistable sheaves) which translates into a relation in $\hkvar\pt\pser\Ga$ (as well as in $\bQ(y)\pser\Ga$ after taking the Poincar\'e polynomials)
\eq{\label{M vs G}
\sum_{\ga\in\Ga}\cI^\rM_\ga z^\ga
=\prod_{\nu}\rbr{\sum_{\ga\in\Ga_\nu}\cI_\ga z^\ga}.
}
By the definition of DT invariants, we have
\eq{
\sum_{\ga\in\Ga}\cI^\rM_\ga z^\ga
=\Exp\rbr{\frac{\sum_{\ga\in\Ga^*}\Om^\rM_\ga z^\ga}{y\inv-y}},\qquad
\sum_{\ga\in\Ga_\nu}\cI_\ga z^\ga
=\Exp\rbr{\frac{\sum_{\ga\in\Ga^*_\nu}\Om_\ga z^\ga}{y\inv-y}}.
}
Therefore we obtain from \eqref{M vs G}
$$
\Exp\rbr{\frac{\sum_{\ga\in\Ga^*}\Om^\rM_\ga z^\ga}{y\inv-y}}
=\prod_\nu
\Exp\rbr{\frac{\sum_{\ga\in\Ga^*_\nu}\Om_\ga z^\ga}{y\inv-y}}
=\Exp\rbr{\frac{\sum_\nu\sum_{\ga\in\Ga^*_\nu}\Om_\ga z^\ga}{y\inv-y}}
$$
and $\Om^\rM_\ga=\Om_\ga$.
\end{proof}

Let us give a slightly different formulation of the above theorem.
Consider the generating functions in
$\bQ\pser{y,t}$ defined by
\eq{
H_{r,\ga_1}(y,t;J)=H_{r,\ga_1}:=\sum_{\ga_2} \cI^\rM(\ga,y;J)\, t^{r\De(\ga)},
}
and 
\eql{
h_{r,\ga_1}(y,t;J)=h_{r,\ga_1}:=\sum_{\ga_2} \Omb(\ga,y;J)\, t^{r\De(\ga)}.
}{hrc1}


For rational and ruled surfaces, we can write explicit formulas for $H_{r,\ga_1}(y,t;J)$.
The following theorem relates these invariants to $h_{r,\ga_1}(y,t;J)$.

\begin{theorem}
\label{th:Htoh}
Assume that $J$ is generic and $J\cdot K_S<0$.
Then, for every $\ta\in\bR$,
\eq{\sum_{\ga_1\cdot J/r=\ta}h_{r,\ga_1}z_0^rz_1^{\ga_1}=\log\left(
1+\sum_{\ga_1\cdot J/r=\ta}H_{r,\ga_1}z_0^rz_1^{\ga_1}
\right),}
which is equivalent to
\eql{
h_{r,\ga_1}
=\sum_{
\over{\sum(r^{(i)},\ga_1^{(i)})=(r,\ga_1)}
{\mu_J(\ga^{(i)})=\mu_J(\ga)\ \forall i}} \frac{(-1)^{\ell-1}}{\ell}
\prod_{i=1}^\ell H_{r^{(i)},\ga_1^{(i)}}.
}{Htoh}
\end{theorem}
\begin{proof}
We can write
$$\sum_{\ga_1\cdot J/r=\ta}H_{r,\ga_1}z_0^{r}z_1^{\ga_1}
=\sum_{\mu_J(\ga)=\ta}\cI^\rM_\ga z_0^{r}z_1^{\ga_1}t^{\ga_1^2/2r-\ga_2}
=\sum_{\mu_J(\ga)=\ta}\cI^\rM_\ga u^rt^{-\ga_2},
$$
$$\sum_{\ga_1\cdot J/r=\ta}h_{r,\ga_1}z_0^{r}z_1^{\ga_1}
=\sum_{\mu_J(\ga)=\ta}\Omb_\ga z_0^{r}z_1^{\ga_1}t^{\ga_1^2/2r-\ga_2}
=\sum_{\mu_J(\ga)=\ta}\Omb_\ga u^rt^{-\ga_2},
$$
where $u=z_0z_1^{\ga_1/r}t^{\ga_1^2/2r^2}$ is independent of $\ga$ for fixed $\mu_J(\ga)=\ta$.
Therefore we have to prove
\eql{
1+\sum_{\mu_J(\ga)=\ta}\cI^\rM_\ga u^rt^{-\ga_2}
=\exp\rbr{\sum_{\mu_J(\ga)=\ta}\Omb_\ga u^rt^{-\ga_2}}.
}{eq:u-q}


Given $\nu\in\bR$ and any class \ga with 
\eql{\ga_1\cdot J/r=\ta,\qquad \ga_2/r=\nu,}{ta-nu}
we obtain from \eqref{reduced Hilb} and the assumption that $J$ is generic that
\eql{p_J(\ga,n)=p_{J,\ta}(n)+\nu,}{p Jta}
where $p_{J,\ta}$ is a polynomial that depends only on $J$ and $\ta$.
Moreover, if $\ga$ satisfies \eqref{p Jta}, then it also satisfies \eqref{ta-nu}.
We can write equation \eqref{M vs G} as
$$1+\sum_{\mu_J(\ga)=\ta}\cI^\rM_\ga z^\ga
=\prod_{\nu}
\rbr{1+\sum_{p_J(\ga)=p_{J,\ta}+\nu}\cI_\ga z^\ga}.
$$
where $z^\ga=z_0^rz_1^{\ga_1}z_2^{\ga_2}$.
On the other hand equation \eqref{Om to I} can be written as
$$1+
\sum_{p_J(\ga)=p}\cI_\ga z^\ga
=\exp\rbr{
\sum_{p_J(\ga)=p}\Omb_\ga z^\ga
}$$
for any polynomial $p$.
Combining these two equations, we obtain
$$1+\sum_{\mu_J(\ga)=\ta}\cI^\rM_\ga z^\ga
=\prod_\nu\exp\rbr{
\sum_{p_J(\ga)=p_{J,\ta}+\nu}\Omb_\ga z^\ga}
=\exp\rbr{\sum_{\mu_J(\ga)=\ta}\Omb_\ga z^\ga}.
$$
Using the substitution $u=z_0z_1^{\ga_1/r}$ and $t=z_2\inv$, we obtain \eqref{eq:u-q}.
%
%
\end{proof}

\begin{remark}
Consider a possibly non-generic polarization $J=\pm K_S$.
Then $\ang{\ga,\ga'}=0$ whenever $\mu_J(\ga)=\mu_J(\ga')$ by \eqref{skew-sim}.
By formula \eqref{reduced Hilb}, we can still write the reduced Hilbert polynomial in the form
\eqref{eq:p J ta} $p_J(\ga,n)=p_{J,\ta}(n)+\frac{\ga_2}r,$
where $\ta=\mu_J(\ga)$.
Assuming that $J\cdot K_S<0$, we still obtain the relation \eqref{M vs G} between Gieseker invariants $\cI_\ga$ and Mumford invariants $\cI^\rM_\ga$.
This formula can be translated into a relation between invariants $\hm_{r,\ga_1}$ and $h_{r,\ga_1}$ similar to \eqref{Htoh}.
More precisely, by equation \eqref{De}, we can write $-\ga_2=r\De(\ga)-\ga_1^2/2r$ and consider the series
\eq{t^{-\ga_1^2/2r}H_{r,\ga_1}(y,t;J)
=\sum_{\ga_2} \cI^\rM(\ga,y;J)\, t^{-\ga_2}
}
which behaves better than $H_{r,\ga_1}$ as the second Chern class respects short exact sequences.
We have an analogue of \eqref{Htoh}
\eq{\label{hnongen}
t^{-\ga_1^2/2r}h_{r,\ga_1}
=\sum_{
\over{\sum(r^{(i)},\ga_1^{(i)})=(r,\ga_1)}
{\mu_J(\ga^{(i)})=\mu_J(\ga)\ \forall i}} \frac{(-1)^{\ell-1}}{\ell}
\prod_{i=1}^\ell t^{-(\ga_1^{(i)})^2/2r^{(i)}}H_{r^{(i)},\ga_1^{(i)}}.
}
\end{remark}

\section{Explicit results for ruled surfaces}
\label{explicitresults}
In this section, we determine in a number of different cases the motivic DT-invariants giving the dimensions of intersection cohomology groups in cases where $M_\gamma$ is singular. For the projective plane, these invariants were computed earlier by G\"ottsche \cite{Gottsche1990} for $r=1$, Yoshioka \cite{Yoshioka1995} for $r=2$, and the first author \cite{Manschot:2011ym, manschot_sheaves} for $r\geq 3$. This section gives explicit results for $\dim \IH^n(M_\ga)$ for the
ruled surfaces $\pi:\Sigma_{g,d} \to C$.

\subsection{Wall-crossing and suitable polarization}
Let $\pi:\Sigma_{g,d} \to C$ be a ruled surface with fiber $f\simeq \bP^1$ over a smooth projective curve $C$ of genus $g$. Here $\Si_{g,d}=\bP(L\oplus O_C)$ with a line bundle $L$ of degree $d\ge0$.
The intersection numbers on $\Sigma_{g,d}$ are 
$$C^2=-d\leq 
0,\qquad C\cdot f=1,\qquad f^2=0,$$
and its canonical class is
$K_{\Sigma_{g,d}}=-2C+(2g-2-d) f$.
We parametrize a polarization $J$ by
$J_{m,n}=m(C+ d f)+nf$ with
$$m=f\cdot J_{m,n}\in\bQ_{\ge0},\qquad 
n=C\cdot J_{m,n}\in\bQ_{\ge0}.$$
Then $\lambda J_{m,n}$ for appropriate $\lambda\in\bZ$ is the first Chern class of a nef line bundle. Note that the requirement $J\cdot K_S<0$ implies the further constraint
$\frac nm\ge g-1-\frac d2$.

We recall the generating functions $\hm_{r,\ga_1}(y,t;J_{0,1})$ for
the boundary polarization $J_{0,1}=f$.

\begin{proposition}[\cf {\cite{Manschot:2011ym, mozgovoy_invariants}}]
\label{propH}
For the ``boundary'' polarization $J_{0,1}$,
$\hm_{r,\ga_1}(y,t;J_{0,1})$ are given for all $r\geq 1$ by  
\be 
\hm_{r,\ga_1}(y,t;J_{0,1})
=\begin{cases}
\hm_{r}(y,t),&\text{if }\ga_1\cdot f=0 \pmod r,\\
0,&\text{if }\ga_1\cdot f\neq 0 \pmod r.
\end{cases}
\ee
with 
\be
\label{HMf}
\begin{split}
\hm_{r}(y,t)=\hm_r:=& (-y)^{-r^2(1-g)} \prod_{n=1}^\infty \frac{(1-y^{-2r+1}t^{n-1})^{2g}(1-y^{2r-1}t^n)^{2g}}{(1-y^{-2r}t^{n-1})\,(1-y^{2r}t^{n})(1-t^n)^2}\\
& \times \prod_{k=1}^{r-1} \frac{(1-y^{-2k+1}t^{n-1})^{2g}(1-y^{2k-1}t^n)^{2g}}{(1-y^{-2k} t^{n-1} )^2\,(1-y^{2k} t^{n} )^2}.
\end{split}
\ee
\end{proposition}

\begin{proof}
By \cite[Corollary 5.2]{mozgovoy_invariants}, if $r\mid \ga_1\cdot f$, then 
$$
\widetilde Z_f(r,\ga_1):=
\sum_{\ga_2} P(\mathfrak{M}_\gamma)\,t^{r\Delta}
=P(\Bun_{C,r})\prod_{k\geq 1}\prod_{i=-r}^{r-1} Z_C(y^{2rk+2i}t^k),
$$
where $\Bun_{C,r}$ is the stack of vector bundles of rank $r$ and degree zero over $C$ that has a Poincar\'e polynomial
\eq{\label{Bun}
P(\Bun_{C,r})=\frac{(1-y)^{2g}}{y^2-1}\prod_{i=1}^{r-1} Z_C(y^{2i}) 
}
and $Z_C(t)$ is the (Poincar\'e polynomial of the) zeta function of the curve $C$ which has the explicit form
\be
\label{zetaC}
Z_C(t)=\frac{(1-yt)^{2g}}{(1-t)(1-y^2t)}.
\ee
The function of our interest is 
\be
\hm_r(y,t)=\sum_{\gamma_2} (-y)^{\chi(\gamma,\gamma)}P(\mathfrak{M}_\gamma) t^{r\Delta}.
\ee
Using that $\chi(\gamma,\gamma)=r^2(-2\Delta(\gamma)+1-g)$, we obtain
$$
\hm_r(y,t)= (-y)^{r^2(1-g)} P(\Bun_{C,r})\prod_{k\geq 1}\prod_{i=-r}^{r-1} Z_C(y^{2i}t^k).
$$
After substitution of \eqref{Bun} and \eqref{zetaC}, this expression is easily rewritten to Eq.~(\ref{HMf}).
\end{proof}

\begin{remark}
\noindent \begin{enumerate}
\item For $r=1$, (\ref{HMf}) agrees with G\"ottsche \cite[Theorem 0.1]{Gottsche1990}, and for $r=2$ with Yoshioka \cite[Theorem 0.1]{Yoshioka1995}. 
\item The generalization of Proposition \ref{propH} to generating
  functions of $E(\ms\ga)$ can be found in \cite[Conjecture
  4.3]{Manschot:2011ym}, where it is also shown that $t^{-\frac{r\chi(S)}{24}}\hm_r(y,t)$ may be written
  in terms of Dedekind eta and Jacobi theta functions. 
\item Haghighat \cite{Haghighat:2015coa}
 provides a string theoretic explanation of the $H_r$ for the surfaces $\Sigma_{1,d}$. 
\end{enumerate}
\end{remark}

To determine $\hm_{r,\ga_1}(y,t;J_{\eps,1})$, we need to subtract from $\hm_{r,\ga_1}(y,t;J_{0,1})$ the contributions due to sheaves with HN-filtrations of length $>1$ for $J_{\eps,1}$. A useful tool for this is the wall-crossing formula of Joyce for $\cI^\rM(\gamma)$ \cite{Joyce:2004tk}, which we now recall. We will state this formula for $\cI^\rM$, although it is more generally applicable.

\begin{definition}
Let $(\ga^{(i)})=(\ga^{(1)}, \ga^{(2)},\dots, \ga^{(\ell)})$ be a tuple of Chern characters with
$\ga^{(i)}=(r^{(i)},\ga_1^{(i)},\ga_2^{(i)})$
and $r^{(i)}\in \bN_{>0}\ \forall i$. We define $S((\ga^{(i)}), J, {J'} )$ as follows. If for all $i=1,\dots,\ell-1$, we have either
\begin{enumerate}
\item[(a)] $\mu_{J}(\ga^{(i)})\leq \mu_{J}(\ga^{(i+1)})$ and $\mu_{J'}(\sum_{j=1}^i\ga^{(j)}) > \mu_{J'}(\sum_{j=i+1}^\ell\ga^{(j)})$, or
\item[(b)] $\mu_{J}(\ga^{(i)})> \mu_{J}(\ga^{(i+1)})$ and $\mu_{J'}(\sum_{j=1}^i\ga^{(j)}) \leq \mu_{J'}(\sum_{j=i+1}^\ell \ga^{(j)})$,
\end{enumerate}
then $S((\ga^{(i)}),J,{J'} )=(-1)^k$ where $k$ is the number of $i=1,\dots,\ell-1$ such that ({a}) is correct. Otherwise, $S((\ga^{(i)}),J,{J'} )=0$.
\end{definition}

Then we have the following theorem of Joyce. 
\begin{theorem}[{\cite[Theorem 6.21]{Joyce:2004tk}}]
\label{JoyceWC}
Under a change of polarization $J \to {J'}$ the invariants $\cI^\rM(\gamma,y;{J'})$ are expressed in terms of $\cI^\rM(\gamma,y;{J})$ by
$$
\cI^{\rM}(\gamma,y;{J'})
=\sum_{\over{\sum_{i=1}^\ell \gamma^{(i)}=\gamma,}{r^{(i)}\geq 1\,\forall i}}
S\rbr{(\ga^{(i)}),J,J'}\,(-y)^{-\sum_{j<i}(r^{(j)}\ga_1^{(i)}-r^{(i)}\ga_1^{(j)})\cdot K_{S}} \prod_{i=1}^\ell \cI^{\rM}(\gamma^{(i)},y;{J}).
$$
\end{theorem}

Our first aim is to determine a generating function for the invariants
$\cI^{\rM}(\gamma,y;{J})$ with the polarization $J$ sufficiently close
to $J_{0,1}=f$, such that no walls exist between $J$ and
$f$. Such a polarization clearly depends on $\ga$ which is made
precise in the following definition and proposition following
\cite[Remark 5.3.6]{huybrechts_geometry}. 
\begin{definition}
\label{cond_xi}
A $\gamma$-suitable polarization $J$ is a polarization
such that for any $\xi\in \mathrm{Pic}(S)$ satisfying the following two conditions:
\begin{enumerate}
\item $\xi^2$ is bounded by
\be
\label{xi2bound}
-\frac{r^4}{2} \Delta(\ga)\leq  \xi^2<0,
\ee
\item 
either $\xi\cdot f=0$ or $(\xi\cdot f)(\xi\cdot J)>0$.
\end{enumerate}
\end{definition}

\begin{proposition}
No walls exist between $f$ and a $\gamma$-suitable polarization $J$.
\end{proposition}

\begin{proof}
From Equations (\ref{mustab}) and (\ref{Gstab}) we deduce that a wall for $\gamma$ exists between $f$ and $J$ if there exist
$\ga^{(1)}$ and $\ga^{(2)}$ such that $(r^{(1)}\ga_1^{(2)}-r^{(2)}\ga_1^{(1)})\cdot f$ and
$(r^{(1)}\ga_1^{(2)}-r^{(2)}\ga_1^{(1)})\cdot J$ have
a different sign. We set $r^{(1)}\ga_1^{(2)}-r^{(2)}\ga_1^{(1)}=\xi\in
H^2(S,\mathbb{Z})$. Then from (\ref{Delta_filtration}) we find that
\be
r\Delta(\ga)=\sum_{i=1,2} r^{(i)}\Delta^{(i)} -\frac{1}{2r^{(1)}r^{(2)}r} \xi^2.
\ee
By the Bogomolov inequality $r\Delta \geq 0$, and therefore
we arrive at
$$
-2r^2r^{(1)}r^{(2)}\Delta(\ga)\leq \xi^2.
$$
The left hand side is minimized by $r^{(1)}=r^{(2)}=r/2$. Moreover, $\xi$ is
negative definite, $\xi^2<0$, since $\xi\cdot f$ and $\xi\cdot J$ have
a different sign and the signature of $H^2(S,\mathbb{Z})$ is
$(1,b_2(S)-1)=(1,1)$. Thus, the $\xi$ satisfy Condition (1) in
Definition \ref{cond_xi}. However, the different sign violates
Condition (2) and therefore no walls exist between $f$ and a suitable
polarization $J$.
\end{proof}

The next proposition gives a closed
expression for the generating function for invariants \,
$\cI^{\rM}(\ga,y;J)$ for a suitable polarization
$J=J_{\varepsilon,1}$. Since the generating function sums over all
$\ga_2$, the choice of suitable polarization $J_{\varepsilon,1}$ is determined as follows. Truncate $\hm_{r,\ga_1}(y,t,J_{\varepsilon,1})$ at
some power $t^K$, with $K$ the largest value of $r\Delta$ of interest, and denote the corresponding \ga by $\gamma_{\rm max}$. Then
$J_{\varepsilon,1}$ is chosen such that it is $\gamma_{\rm
  max}$-suitable, which implies by Equation (\ref{xi2bound}) that $J_{\varepsilon,1}$ is $\gamma$-suitable
for the terms of $\hm_{r,\ga_1}$ with $r\Delta(\ga)<K$. 

We have the following proposition
\begin{proposition}
\label{HMresummed}
Assume $|y|< 1$ and $\ga_1=\beta C-\alpha f$, then $\hm_{r,\ga_1}(J_{\varepsilon,1})$ equals
\begin{multline}
\hm_{r,\ga_1}(J_{\varepsilon,1})=\\
\begin{cases}
\sum_{r^{(1)}+\dots+r^{(\ell)}=r} \frac{y^{2\sum_{i=2}^\ell (r^{(i)}+r^{(i-1)})\{\frac{\alpha}{r}\sum_{k=i}^\ell r^{(k)}\}}}{\prod_{i=2}^\ell\left(1-y^{2(r^{(i)}+r^{(i-1)})}\right)} \prod_{i=1}^\ell \hm_{r^{(i)}},
&\text{if }\ga_1\cdot f=0 \pmod r,\\ 
0,&\text{if }\ga_1\cdot f\neq 0 \pmod r,
\end{cases}
\end{multline}
where $\{ x \}=x-\lfloor x \rfloor$ is the fractional part of $x$.
\end{proposition}

\begin{proof} The proof follows the proof of \cite[Proposition 4.1]{manschot_sheaves}. We substitute the wall-crossing formula of Theorem \ref{JoyceWC} with $J=J_{0,1}$ and $J'=J_{\varepsilon,1}$ in $\hm_{r,\ga_1}(J_{\varepsilon,1})$,
\begin{multline}
\label{HsubWC}
\hm_{r,\ga_1}(J_{\varepsilon,1})\\
=\sum_{\ga_2}
\sum_{
\over{\sum_{i=1}^\ell \gamma^{(i)}=\gamma,}
{r^{(i)}\geq 1\,\forall i}}
S((\ga^{(i)}),J_{0,1},{J_{\varepsilon,1}} )\,(-y)^{-\sum_{j<i}(r^{(j)}\ga_1^{(i)}-r^{(i)}\ga_1^{(j)})\cdot K_{S}} \prod_{i=1}^\ell \cI^{\rM}(\gamma^{(i)},y;{J})\,t^{r\Delta}.
\end{multline}
To evaluate the sum we parametrize the first Chern classes $\ga_1^{(i)}$ as $\ga^{(i)}_1=b^{(i)}C-a^{(i)}f$, such that $\sum_{i=1}^\ell a^{(i)}=\alpha$ and $\sum_{i=1}^\ell b^{(i)}=\beta$. 
Then we have from Theorem \ref{JoyceWC}  that $S((\gamma^{(i)}), J_{0,1},J_{\varepsilon,1})$ is non-vanishing if for all $i=1,\dots,\ell$, we have 
\begin{enumerate}
\item[(a)] either $$\frac{b^{(i)}}{r^{(i)}} \leq \frac{b^{(i+1)}}{r^{(i+1)}}\,\,\quad {\rm and}\,\, \quad \frac{ \sum_{j=1}^i b^{(j)}-\varepsilon a^{(j)}}{ \sum_{j=1}^i r^{(j)}}> \frac{\sum_{j=i+1}^\ell b^{(j)}-\varepsilon a^{(j)}}{\sum_{j=i+1}^\ell r^{(j)}},$$
\item[(b)] or  $$\frac{b^{(i)}}{r^{(i)}} > \frac{b^{(i+1)}}{r^{(i+1)}}\,\,\quad {\rm and}\,\, \quad \frac{\sum_{j=1}^i b^{(j)}-\varepsilon a^{(j)}}{\sum_{j=1}^i r^{(j)}}\leq \frac{\sum_{j=i+1}^\ell b^{(j)}-\varepsilon a^{(j)}}{\sum_{j=i+1}^\ell r^{(j)}}.$$
\end{enumerate}
Since $J_{\varepsilon,1}$ is a suitable polarization we deduce that $S((\gamma^{(i)}), J_{0,1},J_{\varepsilon,1})$ can only be non-vanishing if $\frac{b^{(i)}}{r^{(i)}}=\frac{\beta}{r}$ for all $i=1,\dots,\ell$. If in addition $ \frac{\sum_{j=1}^i a^{(j)}}{\sum_{j=1}^i r^{(j)}}< \frac{\sum_{j=i+1}^\ell a^{(j)}}{\sum_{j=i+1}^\ell r^{(j)}}$ for all $i$, then $S((\gamma^{(i)}), J_{0,1},J_{\varepsilon,1})=(-1)^{\ell-1}$. Thus we find in particular that for $\gamma_1\cdot f\neq 0 \mod r$, $\hm_{r,\ga_1}(y,t,J_{\varepsilon,1})=0$.

Next we make the change of variables from $a^{(i)}$ to $s^{(i)}$ defined by
$$a^{(i)}=s^{(i)}-s^{(i+1)},\quad i=1,\dots,\ell-1,\qquad a^{(\ell)}=s^{(\ell)},$$
or inversely $s^{(i)}=\sum_{j=i}^\ell a^{(j)}$. Eliminating $a^{(1)}$ using $a^{(1)}=\alpha-s^{(2)}$, we find that the summation in Eq. (\ref{HsubWC}) reduces to 
all $s^{(i)},\,i\geq 2$ such that $s^{(i)}>\frac{\sum_{j=i}^\ell r^{(j)}}{r} \alpha$. The exponent of $y$ in (\ref{HsubWC}) in terms of the new variables becomes
\be
\begin{split}
-\sum_{j<i}(r^{(j)}\ga_1^{(i)}-r^{(i)}\ga_1^{(j)})\cdot K_{S}&=-2\sum_{j<i} (r^{(j)} a^{(i)}-r^{(i)}a^{(j)}) \\
&=2\alpha (r-r^{(1)})-2\sum_{j=2}^\ell (r^{j}+r^{j-1})s^{(j)}. 
\end{split}
\ee

To evaluate the sum for $\gamma_1\cdot f=0\pmod r$, we first note that (\ref{Delta_filtration}) simplifies in the present situation to $r\Delta=\sum_{i=1}^\ell r^{(i)}\Delta^{(i)}$. 
Assuming that $|y|<1$, the geometric sums over $s^{(i)}$ can be carried out, such that
$$ 
\hm_{r,\ga_1}(J_{\varepsilon,1})
=\sum_{r^{(1)}+\dots+r^{(\ell)}=r}(-1)^{\ell-1} \frac{y^{2\alpha (r-r^{(1)})-2\sum_{i=2}^\ell (r^{(i)}+r^{(i-1)})(1+\lfloor \sum_{j=i}^\ell r^{(j)}\frac{\alpha}{r}  \rfloor)}}{\prod_{i=2}^\ell\left(1-y^{-2(r^{(i)}+r^{(i-1)})}\right)}
\cdot\prod_{i=1}^\ell \hm_{r^{(i)}}.
$$
After multiplication of numerator and denominator by $\prod_{i=2}^\ell -y^{-2(r^{(i)}+r^{(i-1)})} $ and using the identity $(r-r_1)r=\sum_{i=2}^\ell(r^{(i)}+r^{(i-1)}) \sum_{k=i}^\ell r^{(k)}$, we arrive at the desired result.
\end{proof}

\begin{remark}
$\hm_{r,\ga_1}(y,t,J_{\varepsilon,1})$ can be analytically continued beyond $|y|=1$.
\end{remark}


\subsection{Rank 2}
In this subsection, we apply the formulas discussed above to determine $\dim \IH^n(M_\ga)$ for rank 2 sheaves in a number of cases.
Considering $({r,\ga_1})=(2,0)$, Proposition \ref{HMresummed} gives for $\hm_{2,0}$
\be
\hm_{2,0}(J_{\eps,1})=\hm_2+\frac{1}{1-y^4}\hm_1^2.
\ee
Since the suitable polarization $J_{\varepsilon,1}$ is generic, we determine $h_{2,0}$ using (\ref{Htoh}),
\be
\label{h20eps}
h_{2,0}(J_{\eps,1})=\hm_{2,0}(J_{\eps,1})-\frac{1}{2} \hm_{1}^2.
\ee
Following (\ref{bOmtoOm}), the generating function of $\dim
\IH(M_\ga)$, $\sum_{\ga_2} \Omega_\ga\,t^{r\Delta(\ga)}$, is then given by
$$
(y^{-1}-y)\left(h_{2,0}(y,t;J_{\eps,1})-\frac{1}{2}H_{1}(y^2,t^2)\right).
$$
We list $b_n:=\dim \IH^n(M_\gamma)$ 
and numerical DT invariants $\omn_\ga=\Om(\ga,-1;J)$ for $S=\Sigma_{0,d}$
and for small $\ga_2$ in Table~\ref{tab:DT2}. Note that for a suitable
polarization the $\Om_\ga$ are independent of $d$. The numbers are listed up to $\dim_\mathbb{C} {M_\ga}$;
those with $n>\dim_\mathbb{C} {M_\ga}$ are determined by Poincar\'e
duality $b_n=b_{2\dim_\mathbb{C} {M_\ga}-n}$.

\begin{table}[h!]
\begin{tabular}{l|rrrrrrrrrrrrrrrrr}
$\ga_2$ & $b_0$ & $b_2$ & $b_4$ & $b_6$ & $b_8$ & $b_{10}$ & $b_{12}$ & $b_{14}$
& $b_{16}$ & $b_{18}$ & $b_{20}$ & $b_{22}$ & $b_{24}$ & $\omn_\ga$ \\
\hline
2 & 1 & 2 & 3 &  &  &  & & & & & & & &  -12  \\
3 & 1 & 3 & 8 & 16 & 20 &  &  &  &  & & & & &  -96 \\
4 & 1 & 3 & 10 & 24 & 51 & 82 & 103 &  &  &  &  & & &  -548 \\
5 & 1 & 3 & 10 & 26 & 62 & 130 & 232 & 348 & 420 & & & & &  -2464 \\
6 & 1 & 3 & 10 & 26 & 65 & 144 & 301 & 555 & 913 & 1284 & 1518 & & &  -9640 \\
7 & 1 & 3 & 10 & 26 & 65 & 147 & 318 & 642 & 1203 & 2065 & 3172 & 4280 & 4964 & -33792 
\end{tabular}
\vspace{.2cm}
\caption{Table with $b_n$ (with $n\leq
  \dim_\mathbb{C} {M}_{\ga}$) and the numerical DT invariant $\omn$ of $J_{\varepsilon,1}$-semi-stable sheaves
  on $\Sigma_{0,d}$ with $r=2$, $\ga_1=0$, and $2\leq \ga_2\leq 7$.
  }
\label{tab:DT2}
\end{table}

Tensoring a sheaf $F$ on $\Sigma_{g,d}$ with a line bundle with $\ga_1=\ga_2=0$ does not
change $\ga(F)$. As a result, the moduli space $M_\ga$ is a fibration
with fibre the moduli space of such line bundles, i.e. the Jacobian of $C$.
This further implies that the intersection Poincar\'e polynomial involves
a factor $(1-y)^{2g}$. To concisely tabulate the motivic DT
invariants, we define a new set of numbers $b_n'$ through 
\be 
\label{def_bnp}
\IP(M_\gamma)=:(1-y)^{2g} \sum_{n=0}^{2\dim M_\ga-2g} b_n'\,(-y)^n.
\ee
We list in the Tables \ref{tab:DTg=1} and \ref{tab:DTg=2} below the
$b_n'$ for $g=1$ and $g=2$ for $n\leq \dim_\mathbb{C} {M_\ga}-g$. The
numbers $b_n'$ with $n>\dim_\mathbb{C} {M_\ga}-g$ are again determined by Poincar\'e duality.


\begin{table}[h!]
\begin{tabular}{l|rrrrrrrrrrrrrrrrr}
$\ga_2$ & $b_0'$ & $b_1'$ & $b_2'$ & $b_3'$ & $b_4'$ & $b_{5}'$ & $b_{6}'$ & $b_{7}'$
& $b_{8}'$ & $b_{9}'$ & $b_{10}'$ & $b_{11}'$ & $b_{12}'$   \\
\hline
1 & 1 & 2 & 2 & 2 & 2    \\
2 & 1 & 2 & 4 & 10 & 17 & 24 & 30  & 32  & 32 & &  \\
3 & 1 & 2 & 4 & 10 & 21 & 40 & 68 & 108 & 163 & 218 & 256 & 278 & 286 \\
\end{tabular}
\vspace{.2cm}
\caption{Table with $b_n'$ (\ref{def_bnp}) of $J_{\varepsilon,1}$-semi-stable sheaves
  on $\Sigma_{1,d}$ with $r=2$, $\ga_1=0$, and $1\leq \ga_2\leq 3$. }  
\label{tab:DTg=1} 
\end{table}

\begin{table}[h!]
\begin{small}
\begin{tabular}{l|rrrrrrrrrrrrrrrrr}
$\ga_2$ & $b_0'$ & $b_1'$ & $b_2'$ & $b_3'$ & $b_4'$ & $b_{5}'$ & $b_{6}'$ & $b_{7}'$
& $b_{8}'$ & $b_{9}'$ & $b_{10}'$ & $b_{11}'$ & $b_{12}'$ &  $b_{13}'$ & $b_{14}'$ & $b_{15}'$  \\
\hline 
0 & 1 & 0 & 1   \\
1 & 1 & 4 & 3 & 12 & 21 & 20 & 23 & 24  \\
2 & 1 & 4 & 9 & 20 & 48 & 80 & 139 & 224 & 304 & 364 & 387 & 408 & \\
3 & 1 & 4 & 9 & 24 & 60 & 124 & 234 & 432 & 762 & 1216 & 1820 & 2600 & 3359 & 3904 & 4251 & 4384 \\
\end{tabular}
\end{small}
\vspace{.2cm} 
\caption{Table with $b_n'$ (\ref{def_bnp}) of $J_{\varepsilon,1}$-semi-stable sheaves
  on $\Sigma_{2,d}$ with $r=2$, $\ga_1=0$, and $0\leq \ga_2\leq 3$.}  
\label{tab:DTg=2}
\end{table} 

For other polarizations, we can determine $H_{2,\gamma_1}(
J_{m,n})$ using the wall-crossing formula of Theorem
\ref{JoyceWC}. Without loss of generality, we can set $\ga_1=\beta C-\alpha f$ with $\alpha$
and $\beta$ either 0 or 1. For $r=2$, we have either $\ell=1$ or $2$.
Setting for
$\ell=2$, $a^{(1)}=-a$ and $a^{(2)}=a+\alpha$, and similarly for
$b^{(1)}$ and $b^{(2)}$, we arrive at 
\be
\begin{split}
H_{2,\ga_1}(J_{m,n})&=H_{2,\ga_1}(J_{\eps,1})\\
&+\tfrac{1}{2}H_{1}^2\sum_{
\over{a\in \IZ+\frac{\alpha}{2}}
{b\in\IZ+\frac{\beta}{2}}} (\sgn(2nb-2ma+v)-\sgn(2b-2a\eps+v))\\
&\qquad \times  y^{-2b(d-2+2g)-4a}t^{d b^2+2ab},
\end{split} 
\ee 
with $0<v\ll \eps$. Here we used the notation $\sgn(x)-\sgn(y)$,
familiar from the theory of indefinite theta functions
\cite{Gottsche1999, MR1623706, Zwegers-thesis}.
Note that for general $J_{m,n}$ the invariants do depend on $d$. 
 
The infinite sum over $a$ is a geometric series and can be
resummed. For example for $\ga_1=0$, one has
\be
H_{2,0}(J_{m,n})=\hm_{2}+\hm_1^2\sum_{\frac{bn-am}{m}\in [0,1)} \frac{y^{-2b(d-2+2g)-4a}\,t^{d b^2+2ab}}{1-y^4\,t^{-2b}}.
\ee
Note that for given $b$, there is only one value of $a$ contributing to the sum on the right hand side.

As an example of a non-suitable polarization we take $J_{6,5}$, which is generic for small $r\Delta$. We list invariants $b_n'$ for $\Sigma_{1,0}$ and $\Sigma_{1,2}$ in the following tables
\begin{table}[h!]
\begin{tabular}{l|rrrrrrrrrrrrrrrrr}
$\ga_2$ & $b_0'$ & $b_1'$ & $b_2'$ & $b_3'$ & $b_4'$ & $b_{5}'$ & $b_{6}'$ & $b_{7}'$
& $b_{8}'$ & $b_{9}'$ & $b_{10}'$ & $b_{11}'$ & $b_{12}'$  \\
\hline 
1 & 1 & 2 & 2 & 2 & 2 & \\
2 & 1 & 2 & 4 & 10 & 18 & 26 & 32 & 34 & 34  & &  \\
3 & 1 & 2 & 4 & 10 & 21 & 40 & 70 & 116 & 179 & 242 & 286 & 310 & 318 
\end{tabular}
\vspace{.2cm}
\caption{Table with $b_n'$ (\ref{def_bnp}) of $J_{6,5}$-semi-stable sheaves
  on $\Sigma_{1,0}$ with $r=2$, $\ga_1=0$, and $1\leq \ga_2\leq 3$.}  
\label{tab:DTg=265}
\end{table}

\begin{table}[h!]
\begin{tabular}{l|rrrrrrrrrrrrrrrrr}
$\ga_2$ & $b_0'$ & $b_1'$ & $b_2'$ & $b_3'$ & $b_4'$ & $b_{5}'$ & $b_{6}'$ & $b_{7}'$
& $b_{8}'$ & $b_{9}'$ & $b_{10}'$ & $b_{11}'$ & $b_{12}'$   \\
\hline 
1 & 1 & 2 & 2 & 2 & 2 &  &     \\
2 & 1 & 2 & 4 & 10 & 17 & 24 & 30 & 32 & 32 &    \\
3 & 1 & 2 & 4 & 10 & 21 & 40 & 68 & 108 & 163 & 218 & 256 & 278 & 286   \\
\end{tabular}
\vspace{.2cm}
\caption{Table with $b_n'$ (\ref{def_bnp}) of $J_{6,5}$-semi-stable sheaves
  on $\Sigma_{1,2}$ with $r=2$, $\ga_1=0$, and $1\leq \ga_2\leq 3$.}  
\label{tab:DTg=265_1}
\end{table}

Finally we consider a non-generic polarization, namely $J=-K_{\Sigma_{g,d}}$. For this choice of polarization the torus is commutative and the invariants $\cI(\gamma,y;J)$ can be related to $\dim \IH(M_\gamma)$. The anti-canonical class $-K_{\Sigma_{g,d}}$ does only lie in the ample cone for $g=0$ and $d=0,1$. For these cases we have $-K_{\Sigma_{0,d}}=J_{2,2-d}$. Since $J_{2,2-d}$ is non-generic, we need to consider in more detail partitions $\sum_{i=1}^2 \ga^{(i)}=\gamma$ of $\gamma$ such that $p_J(\gamma^{(i)})=p_J(\gamma)$ for $i=1,\dots,\ell$. While this implies that $\gamma^{(i)}$ are proportional for generic $J$, for $J=-K_{\Sigma_{0,d}}$ this can also occur for $\gamma^{(i)}$ which are not proportional. Equation (\ref{I to Omb}) shows that we need take these partitions in to account to determine $\bar \Omega_\ga$ from $\mathcal{I}_\ga$. We consider first the case $\Sigma_{0,0}$ together with $\ga_1=0$ such that the slope vanishes. Then the $\gamma_1^{(i)}$ which satisfy $-\ga_1^{(i)}\cdot K_{\Sigma_{0,0}}=0$ are of the form $a^{(i)}(C-f)$. Similarly for $d=1$, the $\ga_1^{(i)}$ are of the form $a^{(i)}(2C-f)$ lead to a vanishing slope. As prescribed by Equation (\ref{hnongen}), we sum over all $a^{(1)}=-a^{(2)}=a\in \mathbb{Z}$ and find that for $d=0,1$ the function $h_{2,0}(y,t;-K_{\Sigma_{0,d}})$ is given by
\be
h_{2,0}(-K_{\Sigma_{0,d}})=\hm_{2,0}(J_{2,2-d})-\tfrac{1}{2} H_{1}^2\,\sum_{a\in \IZ} t^{2(1+3d)a^2}.
\ee 
For this polarization the motivic DT invariants are listed in Tables \ref{tab:DT2KL0} and \ref{tab:DT2KL1}. 

\begin{table}[h!]
\begin{tabular}{l|rrrrrrrrrrrrrrrrr}
$\ga_2$ & $b_0$ & $b_2$ & $b_4$ & $b_6$ & $b_8$ & $b_{10}$ & $b_{12}$ & $b_{14}$
& $b_{16}$ & $b_{18}$ & $b_{20}$ & $b_{22}$ & $b_{24}$ &  $\omn_\ga$ \\
\hline
2 & 1 & 2 & 3 &  &  &  & & & & & & & &  -12  \\
3 & 1 & 3 & 8 & 16 & 20 &  &  &  &  & & & & &  -96 \\
4 & 1 & 3 & 10 & 24 & 51 & 83 & 104 & & &  &  & & & -552 \\
5 & 1 & 3 & 10 & 26 & 62 & 130 & 234 & 354 & 428 & &  & &  &  -2496 \\
6 & 1 & 3 & 10 & 26 & 65 & 144 & 301 & 559 & 927 & 1316 & 1560 & &  &  -9824 \\
7 & 1 & 3 & 10 & 26 & 65 & 147 & 318 & 642 & 1209 & 2091 & 3244 & 4416 & 5140  &  -34624 \\
\end{tabular}
\vspace{.2cm}
\caption{The motivic DT invariants $b_n$ and the
numerical DT invariant $\omn_\ga$ of $J_{1,1}$-semi-stable sheaves
on $\Sigma_{0,0}$ with $r=2$, $\ga_1=0$, and $2\leq \ga_2\leq 7$.}  
\label{tab:DT2KL0} 
\end{table}

\begin{table}[h!]
\begin{tabular}{l|rrrrrrrrrrrrrrrrr}
$\ga_2$ & $b_0$ & $b_2$ & $b_4$ & $b_6$ & $b_8$ & $b_{10}$ & $b_{12}$ & $b_{14}$
& $b_{16}$ & $b_{18}$ & $b_{20}$ & $b_{22}$ & $b_{24}$ &  $\omn_\ga$ \\
\hline
2 & 1 & 2 & 3 &  &  &  & & & & & & & &  -12  \\
3 & 1 & 3 & 8 & 16 & 21 &  &  &  &  & & & & &  -98 \\
4 & 1 & 3 & 10 & 24 & 51 & 84 & 109 & & &  &  & & & -564 \\
5 & 1 & 3 & 10 & 26 & 62 & 130 & 236 & 362 & 449 & &  & &  &  -2558 \\
6 & 1 & 3 & 10 & 26 & 65 & 144 & 301 & 561 & 939 & 1352 & 1634 & &  &  -10072 \\
7 & 1 & 3 & 10 & 26 & 65 & 147 & 318 & 642 & 1212 & 2106 & 3299 & 4551 & 5379  &  -35518 \\
\end{tabular}
\vspace{.2cm}
\caption{The motivic DT invariants $b_n$ and numerical DT invariant
  $\omn_\ga$ of $J_{2,1}$-semi-stable sheaves
  on $\Sigma_{0,1}$ with $r=2$, $\ga_1=0$, and $2\leq \ga_2\leq 7$.}  
\label{tab:DT2KL1} 
\end{table}


\subsection{Rank 3}

As example of higher rank sheaves, we tabulate in this section $\dim \IH(M_\ga)$ with $r=3$ in various cases. First we consider a suitable polarization for $(r,\ga_1)=(3,0)$. Then Proposition \ref{HMresummed} evaluates to
\be
\hm_{3,0}(J_{\varepsilon,1})=\hm_3+\frac{2}{1-y^6} \hm_{1} \hm_{2}+\frac{1}{(1-y^4)^2} \hm_{1}^3.
\ee
Then the generating function
$h_{3,0}(y,t,J_{\varepsilon,1})$ (\ref{hrc1}) follows from
Theorem \ref{th:Htoh} and is given by
\be
h_{3,0}(J_{\varepsilon,1})=\hm_{3,0}(J_{\varepsilon,1})-\hm_1 \hm_{2,0}(J_{\varepsilon,1})+\frac{1}{3}\hm_{1}^3.
\ee
The generating function of $\Om_\ga$ is then given by
$h_{3,0}(y,t,J_{\varepsilon,1})-\frac{1}{3} \hm_{1}(y^3,t^3)$. We list
in Tables \ref{tab:DT3g0}, \ref{tab:DT3g1} and \ref{tab:DT3g2} the
invariants $b_n=\dim \IH^n$ and $b_n'$ defined as in Equation (\ref{def_bnp})

\begin{table}[h!]
\begin{small}
\begin{tabular}{l|rrrrrrrrrrrrrrrrr}
$\ga_2$ & $b_0$ & $b_2$ & $b_4$ & $b_6$ & $b_8$ & $b_{10}$ & $b_{12}$ & $b_{14}$
& $b_{16}$ & $b_{18}$ & $b_{20}$ & $b_{22}$ & $b_{24}$ & $b_{26}$ & $b_{28}$ &  $\omn_\ga$ \\
\hline
3 & 1 & 2 & 5 & 8 & 9 & 10  & & & & & & & &  & & 60  \\
4 & 1 & 3 & 9 & 21 & 44 & 73 & 104 & 122  & 131 & & & & & & &  885 \\
5 & 1 & 3 & 10 & 25 & 60 & 126 & 242 & 414 & 626 & 830 & 969  & 1020 & & & & 7632 \\
6 & 1 & 3 & 10 & 26 & 64 & 142 & 301 & 585 & 1076 & 1820 & 2838  & 4001 & 5104 & 5852 & 6136 & 49782 
\end{tabular}
\end{small}
\vspace{.2cm}
\caption{The motivic DT invariants $b_n$ and numerical DT invariant
  $\omn_\ga$ of $J_{\varepsilon,1}$-semi-stable sheaves
  on $\Sigma_{0,d}$ with $r=3$, $\ga_1=0$, and $3\leq \ga_2\leq 6$.}  
\label{tab:DT3g0} 
\end{table}

\begin{table}[h!]
\begin{tabular}{l|rrrrrrrrrrrrrrrrrr}
$\ga_2$ & $b_0'$ & $b_1'$ & $b_2'$ & $b_3'$ & $b_4'$ & $b_{5}'$ & $b_{6}'$ & $b_{7}'$
& $b_{8}'$ & $b_{9}'$ & $b_{10}'$ & $b_{11}'$ & $b_{12}'$   \\
\hline
1 & 1 & 2 & 2 & 2 & 2 & 2  & 2 & & & & & &   \\
2 & 1 & 2 & 4 & 10 & 18 & 28 & 44 & 62  & 74 & 80 & 84 & 88 & 90  
\end{tabular}
\vspace{.2cm}
\caption{The invariants $b_n'$ of $J_{\varepsilon,1}$-semi-stable sheaves
  on $\Sigma_{1,d}$ with $r=3$, $\ga_1=0$, and $ \ga_2=1,2$.}  
\label{tab:DT3g1} 
\end{table}

\begin{table}[h!]
\begin{tabular}{l|rrrrrrrrrrrrrrrrrr}
$\ga_2$ & $b_0'$ & $b_1'$ & $b_2'$ & $b_3'$ & $b_4'$ & $b_{5}'$ & $b_{6}'$ & $b_{7}'$
& $b_{8}'$ & $b_{9}'$ & $b_{10}'$ & $b_{11}'$ & $b_{12}'$ & $b_{13}'$ & $b_{14}'$   \\
\hline
0 & 1 & 0 & 1 & 4 & 2 & 4 & 2  & 4 & 3 \\
1 & 1 & 4 & 3 & 12 & 23 & 36 & 67 & 92  & 144 & 196 & 221 & 252 & 264 & 272 & 282   
\end{tabular}
\vspace{.2cm}
\caption{The invariants $b_n'$ of $J_{\varepsilon,1}$-semi-stable sheaves
  on $\Sigma_{2,d}$ with $r=3$, $\ga_1=0$, and $ \ga_2=0,1$.}  
\label{tab:DT3g2} 
\end{table}

Invariants for other values of $J$ can again be determined using the wall-crossing formula. 

\begin{multline*}
S((\gamma^{(i)});J_{0,1},J_{m,n})
=\frac{(-1)^{\ell-1}}{2^{\ell-1}}\prod_{i=2}^\ell \bigg( \sgn(b^{(i)}-b^{(i-1)}+v)\\
-\sgn\bigg(\sum_{j=1}^{i-1} r^{(j)}\sum_{k=i}^\ell [r^{(k)} b^{(k)}n-a^{(k)}m]-\sum_{k=i}^\ell r^{(k)}\sum_{j=1}^{i-1} [r^{(j)}b^{(j)}n-a^{(j)}m] +v\bigg)\bigg)
\end{multline*}

To determine the contribution of partitions of $\gamma$ with $(r^{(1)}, r^{(2)})=(2,1)$, we set $\ga^{(1)}_1=-2b C + a f$ and $\ga_1^{(2)}=2bC-af$
\be
\hm_2 \hm_1\sum_{a,b\in \IZ} (\sgn(2bn-am+v)-\sgn(b+v))\, y^{-6b(2g-2+d)-6a}\, t^{3d b^2+3ab},
\ee
with $0<v\ll 1$. This can be resummed to 
\be
2\, \hm_2 \hm_1\sum_{\frac{2bn-am}{m}\in [0,1)}\frac{y^{-6b(2g-2+d)-6a}\,t^{3d b^2+3ab}}{1-y^6\,t^{-3b}}.
\ee 
The contribution due to $(r^{(1)}, r^{(2)})=(1,2)$ is identical to the above.

For the contribution of partitions with $r^{(i)}=1$ for $i=1,2,3$, we set $\ga^{(1)}=-\ga^{(2)}-\ga^{(3)}$ and $b^{(2)}=b_1$, $b^{(3)}=b_2$, $a^{(2)}=a_1$, $a^{(3)}=a_2$. Then we arrive at 
\be
\begin{split}
\tfrac{1}{4} \hm_1^3\sum_{b_j, a_j\in \IZ}& (\sgn((b_1+b_2)n-(a_1+a_2)m+v)-\sgn(2b_1+b_2+v) )\\
&\times (\sgn(b_2n-a_2m+v)- \sgn(b_2-b_1+v)) \\
&\times y^{-2(b_1+2b_2)(2g-2+d)-4(a_1+2a_2)}\,t^{d (b_1^2+b_2^2+b_1b_2)+2a_1b_1+2a_2b_2+b_1a_2+b_2a_1}.
\end{split}
\ee
This can be resummed to 
\be
\begin{split}
&\hm_1^3 \sum_{\frac{(b_1+b_2)n-(a_1+a_2)m}{m}\in [0,1)}
\sum_{\frac{b_2n-a_2m}{m}\in [0,1)} \\
&\qquad \times \frac{y^{-2(b_1+2b_2)(2g-2+d)-4(a_1+2a_2)}\,t^{d (b_1^2+b_2^2+b_1b_2)+2a_1b_1+2a_2b_2+b_1a_2+b_2a_1}}{(1-y^2t^{-2b_1-b_2})(1-y^2 t^{b_1-b_2})}.
\end{split}
\ee
We list in Tables \ref{tab:DTg=1l=0365} and \ref{tab:DTg=1l=2365} the invariants for the surfaces $\Sigma_{1,0}$ and $\Sigma_{1,2}$ with polarization $J_{6,5}$ 

\begin{table}[h!]
\begin{tabular}{l|rrrrrrrrrrrrr}
$\ga_2$ & $b_0'$ & $b_1'$ & $b_2'$ & $b_3'$ & $b_4'$ & $b_{5}'$ & $b_{6}'$ & $b_{7}'$
& $b_{8}'$ & $b_{9}'$ & $b_{10}'$ & $b_{11}'$ & $b_{12}'$\\
\hline 
1 & 1 & 2 & 2 & 2 & 2 & 2 & 2\\
2 & 1 & 2 & 4 & 10 & 19 & 32 & 52 & 74 & 89  & 96 & 100 & 104 & 106\\  
\end{tabular}
\vspace{.2cm}
\caption{Table with $b_n'$ of $J_{6,5}$-semi-stable sheaves
  on $\Sigma_{1,0}$ with $r=3$, $\ga_1=0$, and $0\leq \ga_2\leq 2$.}  
\label{tab:DTg=1l=0365}
\end{table}

\begin{table}[h!]
\begin{tabular}{l|rrrrrrrrrrrrr}
$\ga_2$ & $b_0'$ & $b_1'$ & $b_2'$ & $b_3'$ & $b_4'$ & $b_{5}'$ & $b_{6}'$ & $b_{7}'$
& $b_{8}'$ & $b_{9}'$ & $b_{10}'$ & $b_{11}'$ & $b_{12}'$  \\
\hline 
1 & 1 & 2 & 2 & 2 & 2 & 2 & 2\\
2 & 1 & 2 & 4 & 10 & 18 & 28 & 44 & 62 & 74 & 80 & 84 & 88 & 90\\
\end{tabular}
\vspace{.2cm}
\caption{Table with $b_n'$ (\ref{def_bnp}) of $J_{6,5}$-semi-stable sheaves
  on $\Sigma_{1,2}$ with $r=3$, $\ga_1=0$, and $1\leq \ga_2\leq 3$.}  
\label{tab:DTg=1l=2365}
\end{table}

Finally, we consider the polarization $J=-K_{\Sigma_{0,d}}$ for $S=\Sigma_{0,d}$ with
$d=0,1$. As in the case of $r=2$, sheaves with equal slope do not
necessarily have proportional Chern character for
this non-generic polarization. Besides sheaves
with Chern character proportional to $(3,0,\ga_2)$, all sheaves with
$\gamma_1^{(i)}=a^{(i)}(C-f)$, have slope
0 for $d=0$. For $r^{(i)}=2$, we need to distinguish between $a^{(i)}$ even and odd,
since the invariants differ in these cases. For $d=1$ the same
comments apply except that in this case sheaves with
$\gamma_1^{(i)}=a^{(i)}(2C-f)$ have a vanishing slope.

Subtracting the contributions of these sheaves from $\hm_{3,0}(J_{2,2-d})$ as in Equation (\ref{hnongen}), we arrive at
$h_{3,0}(J_{2,2-d})$ for $d=0,1$
\be
\begin{split}
h_{3,0}(J_{2,2-d})=&\hm_{3,0}(J_{2,2-d})-\hm_1\, \hm_{2,0}(J_{2,2-d})\, \sum_{a\in \IZ} t^{6(1+3d)a^2 }\\
& -  \hm_1\, \hm_{2,(d-1)C+f}(J_{2,2-d})\, \sum_{a\in \IZ} t^{6(1+3d)(a^2 +a+\frac{1}{4})}\\
&+\frac{1}{3} \hm_{1}^3\sum_{a_1,a_2\in \IZ} t^{2(1+3d)(a_1^2+a_2^2+a_1a_2)} .
\end{split}
\ee
For the first few values of $\ga_2$, the invariants are listed in
Tables \ref{tab:DT3JK0} and \ref{tab:DT3JK1}.


\begin{table}[h!]
\begin{small}
\begin{tabular}{l|rrrrrrrrrrrrrrrrr}
$\ga_2$ & $b_0$ & $b_2$ & $b_4$ & $b_6$ & $b_8$ & $b_{10}$ & $b_{12}$ & $b_{14}$
& $b_{16}$ & $b_{18}$ & $b_{20}$ & $b_{22}$ & $b_{24}$ & $b_{26}$ & $b_{28}$ &  $\omn_\ga$ \\
\hline
3 & 1 & 2 & 5 & 8 & 9 & 10  & & & & & & & &  & & 60  \\
4 & 1 & 3 & 9 & 21 & 44 & 74 & 106 & 124  & 133 & & & & & & &  897 \\
5 & 1 & 3 & 10 & 25 & 60 & 126 & 244 & 421 & 644 & 859 & 1003  & 1056 & & & & 7848 \\
6 & 1 & 3 & 10 & 26 & 64 & 142 & 301 & 588 & 1088 & 1859 & 2931  & 4177 & 5363 & 6162 & 6464 & 51894 
\end{tabular}
\end{small}
\vspace{.2cm}
\caption{The motivic DT invariants $b_n$ and numerical DT invariant
  $\omn_\ga$ of $J_{1,1}$-semi-stable sheaves
  on $\Sigma_{0,0}$ with $r=3$, $\ga_1=0$, and $3\leq \ga_2\leq 6$.}  
\label{tab:DT3JK0} 
\end{table}

\begin{table}[h!]
\begin{small}
\begin{tabular}{l|rrrrrrrrrrrrrrrrr}
$\ga_2$ & $b_0$ & $b_2$ & $b_4$ & $b_6$ & $b_8$ & $b_{10}$ & $b_{12}$ & $b_{14}$
& $b_{16}$ & $b_{18}$ & $b_{20}$ & $b_{22}$ & $b_{24}$ & $b_{26}$ & $b_{28}$ &  $\omn_\ga$ \\
\hline
3 & 1 & 2 & 5 & 8 & 10 & 11  & & & & & & & &  & & 63  \\
4 & 1 & 3 & 9 & 21 & 44 & 75 & 111 & 137  & 149 & & & & & & &  951 \\
5 & 1 & 3 & 10 & 25 & 60 & 126 & 245 & 426 & 665 & 914 & 1107  & 1179 & & & & 8343 \\
6 & 1 & 3 & 10 & 26 & 64 & 142 & 301 & 589 & 1093 & 1880 & 3000  & 4363 & 5753 & 6789 & 7190 & 55218 
\end{tabular}
\end{small}
\vspace{.2cm}
\caption{The motivic DT invariants $b_n$ and numerical DT invariant
  $\omn_\ga$ of $J_{2,1}$-semi-stable sheaves
  on $\Sigma_{0,1}$ with $r=3$, $\ga_1=0$, and $3\leq \ga_2\leq 6$.}  
\label{tab:DT3JK1} 
\end{table}


\providecommand{\bysame}{\leavevmode\hbox to3em{\hrulefill}\thinspace}
\providecommand{\href}[2]{#2}

\end{document}